\documentclass[10pt,a4paper]{article}
\usepackage{commath}
\usepackage{graphicx}
\usepackage{enumitem}
\usepackage{tabularx}
\usepackage{rotating}
\usepackage{multirow}
\usepackage{multicol}
\usepackage{float}
\usepackage{amsthm,amssymb,mathrsfs,amsmath}
\usepackage[bindingoffset=0.2in,left=0.5in,right=0.5in,top=1in,bottom=1in,footskip=.25in]{geometry}
\usepackage{cite}
\usepackage[colorlinks,citecolor=blue,urlcolor=blue,bookmarks=false,hypertexnames=true]{hyperref}
\hypersetup{citecolor=blue}
\newtheorem{definition}{Definition}[section]
\newtheorem{theorem}{Theorem}[section]

\newtheorem{lemma}[theorem]{Lemma}

\newcommand{\y}{\mathsf{y}}
\newcommand{\D}{\mathcal{D}}
\newcommand{ \x}{\mathsf{x}}
\newcommand{\z}{\mathsf{z}}

\title{On a class of coupled fractional nonlinear singular boundary value problems arising in dusty fluid models}
\author{Lok Nath Kannaujiya$^1$\thanks{lok\_2121ma04@iitp.ac.in}, Narendra Kumar$^2$\thanks{narendrakumar.knp@gmail.com}, Amit K. Verma$^3$\thanks{Corresponding author: akverma@iitp.ac.in}\\\small{\it{$^{1,3}$ Department of Mathematics,}} \\\small{\it{Indian Institute of Technology Patna,}}\\\small{\it{ Bihta, Patna 801103, (BR) India.}}\\
\small{\it{$^{2}$ Department of Mathematics,}} \\\small{\it{Indian Institute of Technology Jodhpur,}}\\\small{\it{ Karwar, Jodhpur 342030, India.}}}
\date{}

\begin{document}
\maketitle
\begin{abstract}
In this article, we introduce a new class of coupled fractional Lane-Emden boundary value problems. We employ a novel approach, the fractional Haar wavelet collocation method with the Newton-Raphson method. We analyze the conditions in two cases to present numerical experiments related to the defined system of fractional differential equations. To validate the accuracy of the proposed method we present the convergence of the method, and we demonstrate the method's effectiveness through five numerical experiments, highlighting real-world applications of fractional differential equations. Using figures and tables, we show that the residual error decreases as we increase the value of the maximum level of resolution $J$ while keeping the order of derivatives fixed, and similar trends also observe when $J$ is fixed and vary the order of fractional derivatives. We demonstrate that Mathematica software can be used effectively to solve such nonlinear singular fractional boundary value problems.
\end{abstract}

\section{Introduction}
The development of fractional calculus over the last century has established fractional derivatives as a vital tool for modeling natural phenomena. Many operators are used for fractional differential equations, in which the most useful derivatives are, Grunwald-Letnikov fractional derivative \cite{scherer2011grunwald, miller1993introduction}, Riemann-Liouville fractional derivative (R-L FD) \cite{kilbas2006theory, miller1993introduction} and Caputo derivative \cite{odibat2006approximations, podlubny1999introduction, kumari2024novel}. Numerous studies have investigated the solvability of nonlinear fractional differential equations (FDEs) involving these derivatives \cite{bai2005positive, hajji2014efficient, klimek2018exact, zhang2006existence, ibrahim2007existence}. Researchers \cite{zhang2008eigenvalues, al2010numerical, hajji2014efficient} have also explored fractional versions of singular Sturm-Liouville eigenvalue problems using various analytical and numerical techniques. Ibrahim et al. \cite{ibrahim2007existence} investigate the existence and uniqueness of solutions for multi-order fractional differential equations. Lee et al. \cite{lee2024novel} demonstrate the predictor-corrector method for nonlinear initial value problems and both, linear and nonlinear reaction-diffusion of fractional order.

In this paper, we consider the following a class of coupled fractional Lane-Emden equations using Caputo fractional derivatives,
\begin{equation}\label{P4_problem}
\begin{aligned}
   & \D^{\alpha_1} \y(\x) + \frac{k_1}{\x^{\gamma_1}} \D^{\beta_1} \y(\x) = f_1(\x,\y,\z), \quad 1 < \alpha_1 \leq 2, ~0 < \beta_1 \leq 1, ~k_1 \geq 0, ~\gamma_1>0, \;  \x \in (0,1), \\
&\D^{\alpha_2} \z(\x) + \frac{k_2}{\x^{\gamma_2}} \D^{\beta_2} \z(\x) = f_2(\x,\y,\z), \quad 1 < \alpha_2 \leq 2, ~0 < \beta_2 \leq 1,  ~k_2 \geq 0, ~\gamma_2>0, \;  \x \in (0,1),
\end{aligned}
\end{equation}
subject to the following conditions,
\begin{equation}\label{P4_condition}
\begin{aligned}
     & a\; \y(0) + b\; \y'(0)= \mu_1, \quad c\; \z(0) + d\; \z'(0) =\mu_2,\\
  & \y(1)= \mu_3 \eta_1 \z(\nu_1), \quad \z(1) = \mu_4 \eta_2 \y(\nu_2),
\end{aligned}
\end{equation}
where $f_1(\x,\y,\z), ~f_2(\x,\y,\z)$ are continuous functions, $a, ~b,~ c,~ d,~ \mu_1,~ \mu_2,~ \mu_3,~ \mu_4,$ $\eta_1, ~\eta_2,~ \nu_1 , ~\nu_2$ are non-negative real numbers and $\D^{\alpha_1},~ \D^{\alpha_2},~ \D^{\beta_1},~ \D^{\beta_2}$ are Caputo fractional derivatives. We choose two scenarios of conditions \eqref{P4_condition} to demonstrate the numerical experiments as:\\
\begin{equation}\label{P4_condition_case1}
    \textbf{Case I:} \quad \quad \y'(0)= \frac{1}{b} \left(\mu_1 - a \y(0) \right), \quad  \z'(0) = \frac{1}{d} \left( \mu_2 - c \z(0) \right), \quad \y(1)= \mu_3 \eta_1 \z(\nu_1), \quad \z(1) = \mu_4 \eta_2 \y(\nu_2).
\end{equation}
\begin{equation}\label{P4_condition_case2}
    \textbf{Case II:} \quad \quad \y(0)= \frac{1}{a} \left( \mu_1 -b \y'(0) \right) ,\quad \z(0)= \frac{1}{c} \left(\mu_2 - d \z'(0) \right), \quad \y(1)= \mu_3 \eta_1 \z(\nu_1), \quad \z(1) = \mu_4 \eta_2 \y(\nu_2).
\end{equation}
In Case I \eqref{P4_condition_case1}, If $a = \mu_1 =0$ and $c= \mu_2=0,$ then $\y'(0)=0$ and $\z'(0)=0.$ In Case II \eqref{P4_condition_case2}, we consider two sub-cases. In the first sub-case, if $b=0,$ then $\y(0)= \frac{\mu_1}{a}$ and if $d=0,$ then $\z(0)= \frac{\mu_2}{c}.$ In the second sub-case, if $b=\mu_1=0,$ then $\y(0)=0,$ and if $d=\mu_2=0$ then $\z(0)=0.$

Numerous researchers have explored the solution of Lane-Emden equations \cite{parand2024neural, chandrasekhar1957introduction, chandrasekhar1939book, chambre1952solution, verma2021note, verma2020applications}. For further exploration of numerical and analytical solutions of Lane-Emden equations, readers are encouraged to refer \cite{verma2020review}. For the case $\alpha_1=\alpha_2=2, ~ \beta_1=\beta_2=1$, the proposed problem \eqref{P4_problem} reduce to the classical form of coupled Lane-Emden equations, which have been studied by many scholars \cite{ala2023numerical,rach2014solving,verma2021haar,singh2020solving, kumar2023hybrid}. One of the key uses of the Lane-Emden equation is in studying the transport and chemical reactions of species within a porous catalytic particle. In mixtures with multiple components, species transport inside a porous catalyst cannot simply be described using Fick's law instead, it should follow the Maxwell-Stefan equations \cite{flockerzi2011coupled}. Applying these equations to mass transport in porous solids results in the Dusty Gas Model \cite{jackson1977transport}. In this model, the solid catalyst is treated like "dust," meaning it's made up of large, stationary molecules that interact with the fluid species moving through the catalyst's porous structure. This model captures three types of transport: bulk and Knudsen diffusion, surface diffusion, and viscous Poiseuille-type flow. Originally developed for ideal gas mixtures, the Dusty Gas Model was later extended to non-ideal fluid mixtures, becoming the Dusty Fluid Model. \cite{higler2000nonequilibrium, krishna1997maxwell, verma2021haar, xie2023effective} provides more information on the coupled Lane-Emden equations in the context of these dusty fluid models.

 In 1981, Anderson et al. \cite{anderson1981complementary} and, in 2019 Zahoor et al. \cite{raja2018new} studied the steady state of a spherically symmetrical heat conduction problem accurately describes heat conduction issues affecting the human head. However, the use of a febrifuge, such as acetaminophen or ibuprofen, renders these conditions irrelevant. Wang \cite{wang2020new} introduced a fractional version of the heat conduction problem by modifying the heat conduction equation studied in \cite{anderson1981complementary, raja2018new} to account for the effects of a febrifuge. Some studies have been conducted on fractional Lane-Emden equations \cite{saeed2017haar, wang2020new, lok2024, narendra2025, dehda2024numerical}. The study of coupled fractional-order differential equations is crucial, as such systems frequently emerge in various scientific and engineering applications \cite{Ahmad2023coupled, chen2008numerical, gafiychuk2008mathematical, su2009boundary, ghosh2015solution}. Bai and Fang \cite{bai2004existence} demonstrated the existence of positive solutions for a singular coupled system of fractional order. Wang et al. \cite{wang2018nonlocal} examined a coupled system of fractional differential equations with $m-$point fractional boundary conditions, establishing existence and uniqueness results using conventional fixed-point theorems. In 2019, Zhai and Reh \cite{zhai2019coupled} further contributed to the field by proving the existence and uniqueness of solutions for a novel coupled system of nonlinear fractional differential equations. More recently, Dehda et al. \cite{dehda2024numerical} applied the Haar wavelet method to analyze a regular coupled dynamical system. More details about the system of fractional differential equations can be found in \cite{Rehman2011note, AhmadLuca, xu2014iterative, liu2017bifurcation, zhai2018unique, zhang2016existence, liu2016extremal, su2009boundary, daftardar2004}. However, to the best of our knowledge, the Haar wavelet collocation method has not yet been applied to coupled fractional differential systems with singularities. This gap presents an opportunity for further research in this area.

In recent years, wavelet techniques have gained prominence as efficient tools for solving calculus-related problems. Haar wavelets are favored due to their advantageous properties, such as ease of application, orthogonality, and compact support. The compact support of the Haar wavelet basis allows for the straightforward incorporation of various boundary conditions into numerical algorithms. Since the Haar wavelet basis is linear and piecewise defined, it lacks differentiability. Consequently, rather than differentiation, integration is used to calculate the coefficients. The approximation of orthogonal functions has proven essential in applications such as parameter identification analysis and optimal control \cite{chen1997haar, chen2010wavelet, cattani2004haar, shiralashetti2020haar}. A key advantage of this approach is its ability to transform differential equations into algebraic systems. The Haar wavelet approach for solving linear and nonlinear differential equations, fractional differential equations, integral equations, and integro-differential equations was introduced by Lepik \cite{lepik2005numerical,lepik2007application, lepik2007numerical}. Hariharan et al. \cite{hariharan2009haar} employed the Haar wavelet approach to solve Fisher's equation. This approach is notably efficient for handling boundary value problems with a small number of grid points. Recent studies have explored the Haar wavelet method for tackling fractional differential equations with singularities. In \cite{lok2024}, authors introduced both uniform and non-uniform Haar wavelet methods to address fractional initial value problems with singularities. Izadi et al. \cite{Izadi_2024} investigated a fractional singular boundary value problem (BVP) incorporating Liouville-Caputo fractional derivatives and mixed boundary conditions, employing the quasilinearization method in conjunction with generalized Genocchi polynomials. Narendra et al. \cite{narendra2025} focused on solving nonlinear fractional Lane-Emden equations through the uniform fractional Haar wavelet collocation method combined with quasilinearization.

In this article, we introduce a new class of coupled fractional Lane-Emden equations \eqref{P4_problem} with conditions \eqref{P4_condition} and propose a numerical approach called the fractional Haar wavelet collocation method to simulate the numerical solution for the coupled fractional differential equation \eqref{P4_problem}. We use the fractional integral of the Haar wavelet to derive the method. The proposed method transforms the coupled Lane-Emden equation into system of non-linear algebraic equations, and then, with the help of the Newton-Raphson method, we capture the numerical solutions. We present the analysis to investigate the convergence of the method.  We conduct five experiments to show the effectiveness of our method. At $\alpha_1=\alpha_2=2, ~\beta_1=\beta_2 =1,$ the proposed problem \eqref{P4_problem} reduces to the classical form of coupled Lane Emden equations. We also show two real-world applications of fractional Lane-Emden equations in the dusty fluid model, highlighting their practical use. The proposed approach is novel and opens doors for further research in various directions. We have used Mathematica software to compute all the results of the paper. The work in this paper can be further used to develop Mathematica package.

The structure of this article is as follows: Section \ref{prelim} introduces the fundamental concepts of multiresolution analysis, Haar wavelets, and fractional derivatives. Section \ref{P4_collocation_method} presents the fractional Haar wavelet collocation method along with its algorithmic implementation. Section \ref{P4_convergence} discusses the convergence of the proposed method. Section \ref{P4_Examples} provides numerical experiments to illustrate its effectiveness. Finally, we present the conclusion in section \ref{P4_conclusion}.

\section{Preliminary}\label{prelim}
In this section, we present multiresolution analysis, fractional derivatives, as well as Haar wavelets and their fractional integrals.

\subsection{Multiresolution Analysis (MRA)}
Multiresolution analysis (MRA) is a mathematical framework through which signals or functions can be analyzed on multiple scales. It is one of the components of wavelet theory, and it also helps in constructing wavelet bases. MRA achieves a hierarchical decomposition of the signal that captures both detailed and coarse features of the signal. This method is effective for capturing localized or transient features of a signal like spikes or discontinuities, which Fourier methods cannot adequately represent. One of the most elementary and illustrative examples of MRA is the Haar wavelet, which is widely recognized as the first wavelet and serves as a main example of MRA. Haar MRA is widely known for its primary attributes – compact support, orthogonality, density, and, most importantly, for its contribution to wavelet-based signal analysis \cite{wang2010new, sweldens1998}.
\begin{definition}\label{MRA_def}
    \cite{mallat1989, priyadarshi2018wavelet}  The Haar multiresolution analysis of $L^2(\mathbb{R})$ with scaling function $\chi_{[0,1]} (\x) = \phi (x)$ consist of a sequence of closed subspaces $\mathsf{V}_j \; (j \in \mathbb{Z})$ of $ L^2 (\mathbb{R})$ satisfying,
    \begin{enumerate}
        \item $\mathsf{V}_j \subset \mathsf{V}_{j+1}, \quad \forall j \in \mathbb{Z}.$
        \item $f(x) \in \mathsf{V}_j \Leftrightarrow f(2x) \in \mathsf{V}_{j+1}, \quad \forall j \in \mathbb{Z}.$
        \item $\cap_{j \in \mathbb{Z}} \mathsf{V}_j = \{ 0 \}.$
        \item $ \overline{\cup_{j \in \mathbb{Z}} \mathsf{V}_j} = L^2 (\mathbb{R}).$
        \item  The function $\phi(x) \in \mathsf{V}_0$ the set $ \{\phi(x-k) : k \in \mathbb{Z}\}$ forms an orthonormal basis for $\mathsf{V}_0.$
    \end{enumerate}
    The sequence of wavelet subspaces $W_j$ of $L^2(\mathbb{R})$ are such that $V_{j+1} = V_j \bigoplus W_j.$ Moreover, closure of $\bigoplus_{j \in \mathbb{Z}} W_j$ is dense in $L^2 (\mathbb{R})$ with respect to $L^2$ norm.
\end{definition}
\begin{theorem}
    Mallat \cite{mallat1989} provides that in an orthogonal multiresolution analysis (MRA) with a scaling function $\phi,$ there is a wavelet $\psi \in L^2 (\mathbb{R})$ such that for each $j \in \mathbb{Z},$ the collection $\{ \psi_{j, k} \}_{k \in \mathbb{Z}}$ forms an orthonormal basis for the space $W_j.$ Hence, the family $\{\psi_{j,k}\}_{j,k \in \mathbb{Z}}$ is an orthonormal basis for $L_2(\mathbb{R}).$
\end{theorem}

\subsection{Haar Wavelet}\label{P4_2.1_uniform}
In 2014, Lepik and Hein \cite{lepik2014haar} employed an equal step size of $\Delta \x = \frac{1}{2M}$ to partition the interval $[0,1]$ into $2M$ sub-intervals. The mother wavelet function of the Haar wavelet, denoted as $h_l (\x)$ for $l > 1$, is defined as follows:
\begin{equation}
h_l (\x)=
 \begin{cases}  1, & \vartheta_1(l) \leq \x < \vartheta_2(l),\\
-1, & \vartheta_2(l) \leq \x < \vartheta_3(l),\\
0, & otherwise,
\end{cases}
\end{equation}
such that the values $\vartheta_1(l), \vartheta_2(l),$ and $\vartheta_3(l)$ are given by,
\begin{equation}\label{P4_lepikdef_cof}
\vartheta_1(l) = 2k \left(\frac{M}{m}\right) \Delta \x, \quad
\vartheta_2(l) = (2k+1) \left(\frac{M}{m}\right) \Delta \x, \quad \vartheta_3(l) = 2(k+1) \left(\frac{M}{m}\right) \Delta \x,
\end{equation}
where $l= m+k + 1$, and the parameters $J$, $M$, $j$, $m$, and $k$ are defined such that $J$ represents the maximum level of resolution, $M = 2^J$, $j =0,1,\cdots,J$, $m = 2^j$, and $k = 0,1, \cdots, m-1$. For $l=1$, the Haar wavelet function $h_1(\x)$ is defined as,
 \begin{equation}
 h_1(\x) = \begin{cases}  1, & 0 \leq \x <1,\\
 0, &  otherwise.\end{cases}
 \end{equation}
The integral of the Haar function, $\mathcal{P}_{v,l} (\x)$ for $l >1$ is defined as: 
\begin{equation}
\mathcal{P}_{v,l} (\x) = \begin{cases} 
0, & \x < \vartheta_1(l),\\
\frac{1}{v!}[ \x- \vartheta_1(l)]^v,  & \vartheta_1(l) \leq \x \leq \vartheta_2(l),\\
\frac{1}{v!} \{ [ \x- \vartheta_1(l)]^v - 2 [ \x- \vartheta_2(l)]^v \}, ~& \vartheta_2(l) \leq \x \leq \vartheta_3(l),\\
\frac{1}{v!} \{ [ \x - \vartheta_1(l)]^v -2 [ \x- \vartheta_2(l)]^v  + [ \x - \vartheta_3(l)]^v \}, & \x > \vartheta_3(l).
\end{cases}
\end{equation}
For $l=1,$ we have $\vartheta_1(1) =0, ~ \vartheta_2(1) = \vartheta_3(1)=1,$ and the expression for $\mathcal{P}_{v,1}$ becomes,
$$\mathcal{P}_{v,1} = \frac{\x^v}{v!}.$$
We use the collocation method for computation and define collocation points as follows:
\begin{eqnarray}\label{collocation}
&&\eta_{cl} = \frac{\tilde{\eta}_{cl-1} + \tilde{\eta}_{cl}}{2},~~cl= 1,2,\cdots,4M,
\end{eqnarray}
where $\tilde{\eta}_{cl}$ are the grid points and defined by the following equation,
\begin{eqnarray*}
&& \tilde{\eta}_{cl} = cl\Delta \x, ~~cl= 0,1,\cdots,4M.
\end{eqnarray*}
All the Haar wavelet functions are orthogonal to each other, this fundamental property plays a crucial role in making them effective for numerical analysis\cite{chen1997haar},
\begin{equation}\label{P4_haar_ortho}
    \int_0^1 h_l(\x) h_r (\x) d\x = \begin{cases}
        2^{-j}, & l=r,\\
        0, & l\neq r.
    \end{cases}
\end{equation}
\subsection{Some Definitions} Several definitions of fractional derivatives and integrals are available in \cite{SALESTEODORO2019195}. In this article, we use the following two versions.\\
\textbf{Riemann-Liouville Fractional Integral \cite{podlubny1999introduction,  miller1993introduction}:}
           The Riemann-Liouville fractional integral of order $\alpha$ to function $\psi (\varkappa)$ is given as,
\begin{equation}
I^\alpha \psi (\varkappa) = \frac{1}{\Gamma{\alpha}} \int_0^\varkappa (\varkappa-\tau)^{\alpha -1} \psi (\tau) d\tau, ~~~~~\alpha >0.
\end{equation}
\textbf{Caputo Fractional Derivative \cite{podlubny1999introduction, odibat2006approximations, SALESTEODORO2019195}:} The Caputo fractional derivative of a function $\mathfrak{f}(\xi)$ of order $\gamma $, which is absolute continuous differentiable function up to $(n-1)$ on $[a,b]$, is defined as,
         \begin{equation}
             D^{\gamma} \mathfrak{f}(\xi) = \frac{1}{\Gamma(n-\gamma)} \int_{0}^{\xi} (\xi- \tau)^{n-\gamma -1} \frac{d^n}{d\tau^n} \mathfrak{f} (\tau) d\tau,~~\gamma > 0, ~(n-1 < \gamma \leq n, n \in \mathbb{N}).
         \end{equation}
In what follows, we recall some important lemmas from \cite{podlubny1999introduction, narendra2025} that will be used in this paper.
\begin{lemma}\label{P4_subsec2_lemma2_1}
 Let $\alpha >0, ~ \mu >-1$ and $ \mathfrak{f}(\xi) = \xi^\mu,$ then, $$ I^\alpha \mathfrak{f}(\xi) = \frac{\Gamma(\mu +1)}{\Gamma (\mu + \alpha + 1)} \xi^{(\mu + \alpha)}.$$
\end{lemma}
\begin{lemma}\label{P4_subsec2_lemma2_2}
If $ f(\xi) = \mathfrak{c},$ where $\mathfrak{c}$ is constant and $\alpha >0, ~n = \lceil \alpha \rceil, $ then $D^\alpha f(\xi) =0.$
\end{lemma}
\begin{lemma}\label{P4_subsec2_lemma2_3}
Let $\alpha >0, ~n = \lceil \alpha \rceil$ and consider a function $ \mathfrak{f}(\xi)$, then,
$$ I^\alpha~D^\alpha \mathfrak{f}(\xi)= \mathfrak{f}(\xi) - \sum_{k=0}^{n-1} \frac{\x^k}{k!} \mathfrak{f}^{(k)} (0).$$
\end{lemma}
\begin{lemma}\label{P4_lemma4} 
\cite{lok2024, narendra2025} Let $a,b >0$ and $1 < \alpha \leq 2$ then,
$$a^\alpha - b^\alpha \leq (a-b)^\alpha. $$
\end{lemma}
\subsection{Fractional integral of the Haar wavelets}
Saeed et al. \cite{saeed2017haar} show that the fractional integration of order $\upsilon$ for the Haar wavelets is expressed as follows, when $l = 1$:
\begin{equation}\label{P4_subsec2_eqn1}
I^\upsilon h_1 (t) = \frac{t^\upsilon}{\Gamma(\upsilon +1)},
\end{equation}
and, for $l>1,$  
\begin{align}\label{P4_subsec2_eqn2}
  \mathsf{p}_{\upsilon, l} (\x) = I^\upsilon h_l(t) &= \frac{1}{\Gamma(\upsilon)} \int_0^t (\x -s)^{\upsilon-1} h_l (s) ds,\nonumber\\
 &= \frac{1}{\Gamma(\upsilon +1)} \begin{cases}
 0, & 0 <  \vartheta_1(l),\\
(\x -\vartheta_1(l))^\upsilon & \vartheta_1(l) \leq \x < \vartheta_2(l), \\
(\x-\vartheta_1(l))^\upsilon -2 (\x -\vartheta_2(l))^\upsilon , & \vartheta_2(l) \leq \x < \vartheta_3(l),\\
(\x -\vartheta_1(l))^\upsilon -2 (\x -\vartheta_2(l))^\upsilon + (\x -\vartheta_3(l))^\upsilon, & \x \geq \vartheta_3(l),
\end{cases}
\end{align}
where $\vartheta_1(l), ~ \vartheta_2(l), ~\vartheta_3(l)$ are defined as in equation \eqref{P4_lepikdef_cof}. The Haar matrix $H$ is constructed using collocation points \eqref{collocation} where $H(l,cl) = h_l (\eta_{cl} ).$ Furthermore, by replacing the collocation points \eqref{collocation} into equations \eqref{P4_subsec2_eqn1} and \eqref{P4_subsec2_eqn2},  we derive the integration matrix $P$ for fractional orders of the Haar function, resulting in $P_{\upsilon}(l,cl) = \mathsf{p}_{\upsilon, l}(\eta_{cl}),$
$$ P = 
\begin{bmatrix}
\mathsf{p}_{\upsilon, 1}(\eta_{cl}(1)) & \mathsf{p}_{\upsilon,1}(\eta_{cl}(2)) \cdots & \mathsf{p}_{\upsilon, 1}(\eta_{cl}(4M))\\
\mathsf{p}_{\upsilon, 2}(\eta_{cl}(1)) & \mathsf{p}_{\upsilon, 2}(\eta_{cl}(2)) \cdots & \mathsf{p}_{\upsilon, 2}(\eta_{cl}(4M))\\
\vdots  & \vdots  & \vdots\\
\vdots  & \vdots  &\vdots\\
\mathsf{p}_{\upsilon, 2M}(\eta_{cl}(1)) & \mathsf{p}_{\upsilon, 2M}(\eta_{cl}(2)) \cdots & \mathsf{p}_{\upsilon, 2M}(\eta_{cl}(4M))
\end{bmatrix}
.$$

\section{Proposed Method and Algorithm}\label{P4_collocation_method}
In this section, we introduce the fractional Haar wavelet collocation method by examining two distinct cases derived from the conditions in \eqref{P4_condition}, as follows:\\

\noindent \textbf{Case-I:} Let us consider the nonlinear coupled fractional Lane-Emden equations \eqref{P4_problem} with conditions \eqref{P4_condition_case1}. Following \cite{lepik2014haar}, let the higher-order fractional derivative of equation \eqref{P4_problem} into the Haar wavelet series as:
\begin{eqnarray}
    && \label{P4_sec3eqn1} \D^{\alpha_1} \y(\x) = \sum_{l=1}^{2M} a_l  h_l(\x),\\
    && \label{P4_sec3eqn1_1} \D^{\alpha_2} \z(\x) = \sum_{l=1}^{2M} b_l  h_l(\x).
\end{eqnarray}
Taking fractional integral of equation \eqref{P4_sec3eqn1} followed by Lemma \ref{P4_subsec2_lemma2_3} and conditions \eqref{P4_condition_case1}, we get,
\begin{eqnarray}
    &&  I^{\alpha_1} \D^{\alpha_1} \y(\x) = \sum_{l=1}^{2M} a_l I^{\alpha_1} h_l(\x), \nonumber\\
  && \y(\x) =\sum_{l=1}^{2M} a_l I^{\alpha_1} h_l(\x) + \y(0) + \x \y'(0), \nonumber\\
 &&\label{P4_sec3eqn2}  \y(\x) = \sum_{l=1}^{2M} a_l I^{\alpha_1} h_l(\x) + \frac{\mu_1 \x}{b} + \left(1- \frac{a \x }{b} \right) \y(0).
\end{eqnarray}
Similarly,
\begin{equation}\label{P4_sec3eqn3}
     \z(\x) = \sum_{l=1}^{2M} b_l I^{\alpha_2} h_l(\x) + \frac{\mu_2 \x}{d} + \left(1- \frac{c \x }{d} \right) \z(0).
\end{equation}
Now, again from conditions \eqref{P4_condition_case1}, we have $\y(1) = \mu_3 \eta_1 \z(\nu_1), ~ \text{and} ~ \z(1) = \mu_4 \eta_2 \y(\nu_2),$ then using equations \eqref{P4_sec3eqn2} and \eqref{P4_sec3eqn3} we have,
\begin{eqnarray}
    && \label{P4_sec3eqn4} \left(1-\frac{a}{b} \right) \y(0) - \mu_3 \eta_1 \left(1- \frac{c}{d} \nu_1 \right) \z(0) = \mu_3 \eta_1 \sum_{l+1}^{2M} b_l I^{\alpha_2} h_l(\nu_1) + \frac{\mu_2 \mu_3 \nu_1 \eta_1}{d} - \frac{\mu_1}{b} - \sum_{l=1}^{2M} a_l I^{\alpha_1} h_l(1),\\
    && \label{P4_sec3eqn5}  -\mu_4 \eta_2 \left(1-\frac{a \nu_2}{b} \right)\y(0) + \left(1-\frac{c}{d} \right)\z(0) = \mu_4 \eta_2 \sum_{l=1}^{2M}a_l I^{\alpha_1} h_l(\nu_2) + \frac{\mu_1 \mu_4 \eta_2 \nu_2}{b} - \sum_{l=1}^{2M} b_l I^{\alpha_2} h_l(1) - \frac{\mu_2}{d}.
\end{eqnarray}
Solving the above equations \eqref{P4_sec3eqn4} and \eqref{P4_sec3eqn5}, we get,
\begin{equation}\label{P4_sec3eqn6}
\begin{split}
    \y(0) &= \frac{1}{\left[\left(1-\frac{a}{b} \right) \left(1-\frac{c}{d} \right) - \mu_3 \mu_4 \eta_1 \eta_2 \left(1-\frac{a \nu_2}{b} \right) \left(1-\frac{c\nu_1}{d} \right) \right]} \left\{ \left(1- \frac{c}{d} \right) \left[\mu_3 \eta_1 \sum_{l=1}^{2M} b_l I^{\alpha_2} h_l (\nu_1)  + \frac{\mu_2 \mu_3 \nu_1 \eta_1}{d} - \frac{\mu_1}{b} \right. \right.\\
    & \hspace{1cm} \left. \left. - \sum_{l=1}^{2M} a_l I^{\alpha_1} h_l(1) \right] + \mu_3 \eta_1 \left(1- \frac{c \nu_1}{d} \right) \left[\mu_4 \eta_2 \sum_{l=1}^{2M}a_l I^{\alpha_1} h_l(\nu_2) + \frac{\mu_1 \mu_4 \eta_2 \nu_2}{b} - \sum_{l=1}^{2M} b_l I^{\alpha_2} h_l(1) - \frac{\mu_2}{d} \right] \right\},
\end{split}
\end{equation} 
\begin{equation}\label{P4_sec3eqn7}
\begin{split}
    \z(0) &= \frac{1}{\left[ \left(1-\frac{a}{b} \right) \left(1-\frac{c}{d} \right) - \mu_3 \mu_4 \eta_1 \eta_2 \left(1-\frac{a \nu_2}{b} \right) \left(1-\frac{c\nu_1}{d} \right) \right]} \left\{ \mu_4 \eta_2 (1-\frac{a\nu_2}{b}) \left[\mu_3 \eta_1 \sum_{l=1}^{2M} b_l I^{\alpha_2} h_l(\nu_1) + \frac{\mu_2 \mu_3 \eta_1 \nu_1}{d} \right. \right.\\ 
    & \hspace{0.5cm} \left. \left.  - \frac{\mu_1}{b} - \sum_{l=1}^{2M} a_l I^{\alpha_1} h_l(1) \right] + (1-\frac{a}{b}) \left[\mu_4 \eta_2 \sum_{l=1}^{2M} a_l I^{\alpha_1} h_l(\nu_2) + \frac{\mu_1 \mu_4 \eta_2 \nu_2}{b} - \sum_{l=1}^{2M} b_l I^{\alpha_2} h_l(1) - \frac{\mu_2}{d} \right] \right\}.
    \end{split}
\end{equation}
Substituting the values $\y(0),$ $\z(0)$ from equations \eqref{P4_sec3eqn6} and \eqref{P4_sec3eqn7} into the equations \eqref{P4_sec3eqn2} and \eqref{P4_sec3eqn3}, we get,
\begin{equation}\label{P4_sec3eqn8}
    \begin{split}
        \y(\x) &= \sum_{l=1}^{2M} a_l I^{\alpha_1} h_l(\x) + \frac{\mu_1 \x}{b} +   \frac{\left(1- \frac{a \x }{b} \right)}{\left[ \left(1-\frac{a}{b} \right) \left(1-\frac{c}{d} \right) - \mu_3 \mu_4 \eta_1 \eta_2 \left(1-\frac{a \nu_2}{b} \right) \left(1-\frac{c\nu_1}{d} \right) \right]} \\ & \hspace{0.5cm} \left\{ \left(1- \frac{c}{d} \right) \left[\mu_3 \eta_1 \sum_{l=1}^{2M} b_l I^{\alpha_2} h_l (\nu_1)  + \frac{\mu_2 \mu_3 \nu_1 \eta_1}{d} - \frac{\mu_1}{b} - \sum_{l=1}^{2M} a_l I^{\alpha_1} h_l(1) \right] + \mu_3 \eta_1 \left(1- \frac{c \nu_1}{d} \right) \right. \\& \hspace{1cm} \left.  \left[\mu_4 \eta_2 \sum_{l=1}^{2M}a_l I^{\alpha_1} h_l(\nu_2) + \frac{\mu_1 \mu_4 \eta_2 \nu_2}{b} - \sum_{l=1}^{2M} b_l I^{\alpha_2} h_l(1) - \frac{\mu_2}{d} \right]\right\},
    \end{split}
\end{equation}
\begin{equation}\label{P4_sec3eqn9}
    \begin{split}
        \z(\x) &= \sum_{l=1}^{2M} b_l I^{\alpha_2} h_l(\x) + \frac{\mu_2 \x}{d} +  \frac{\left(1- \frac{c \x }{d} \right)}{\left[ \left(1-\frac{a}{b} \right) \left(1-\frac{c}{d} \right) - \mu_3 \mu_4 \eta_1 \eta_2 \left(1-\frac{a \nu_2}{b} \right) \left(1-\frac{c\nu_1}{d} \right) \right]}  \\& \hspace{0.5cm} \left\{ \mu_4 \eta_2 \left(1-\frac{a\nu_2}{b} \right) \left[\mu_3 \eta_1 \sum_{l=1}^{2M} b_l I^{\alpha_2} h_l(\nu_1) + \frac{\mu_2 \mu_3 \eta_1 \nu_1}{d} - \frac{\mu_1}{b} - \sum_{l=1}^{2M} a_l I^{\alpha_1} h_l(1) \right] + \left(1-\frac{a}{b}\right) \right.\\  & \hspace{0.5cm} \left. \left[\mu_4 \eta_2 \sum_{l=1}^{2M} a_l I^{\alpha_1} h_l(\nu_2) + \frac{\mu_1 \mu_4 \eta_2 \nu_2}{b} - \sum_{l=1}^{2M} b_l I^{\alpha_2} h_l(1) - \frac{\mu_2}{d} \right] \right\}.
    \end{split}
\end{equation}
Taking lower order fractional derivative of equations \eqref{P4_sec3eqn8} and \eqref{P4_sec3eqn9}, we get,
\begin{equation}\label{P4_sec3eqn9_1}
\begin{split}
    \D^{\beta_1} \y(\x) &= \sum_{l=1}^{2M} a_l I^{\alpha_1 - \beta_1} h_l(\x) + \frac{\mu_1}{b} \left(\frac{\Gamma{(2)}}{\Gamma{(2 - \beta_1)}}\right) \x^{1-\beta_1}  \\&  - \frac{a}{b} \left(\frac{\Gamma{(2)}}{\Gamma{(2 - \beta_1)}}\right)  \frac{\x^{1-\beta_1}}{\left[ \left(1-\frac{a}{b} \right) \left(1-\frac{c}{d} \right) - \mu_3 \mu_4 \eta_1 \eta_2 \left(1-\frac{a \nu_2}{b} \right) \left(1-\frac{c\nu_1}{d} \right) \right]} \\& \hspace{0.3cm} \left\{ (1- \frac{c}{d}) \left[\mu_3 \eta_1 \sum_{l=1}^{2M} b_l I^{\alpha_2} h_l (\nu_1)  + \frac{\mu_2 \mu_3 \nu_1 \eta_1}{d}   - \frac{\mu_1}{b} - \sum_{l=1}^{2M} a_l I^{\alpha_1} h_l(1) \right] + \mu_3 \eta_1 (1- \frac{c \nu_1}{d})  \right. \\& \hspace{0.7cm} \left. \left[\mu_4 \eta_2 \sum_{l=1}^{2M}a_l I^{\alpha_1} h_l(\nu_2) + \frac{\mu_1 \mu_4 \eta_2 \nu_2}{b} - \sum_{l=1}^{2M} b_l I^{\alpha_2} h_l(1) - \frac{\mu_2}{d} \right] \right\},
    \end{split}
\end{equation}

\begin{equation}\label{P4_sec3eqn9_2}
    \begin{split}
        \D^{\beta_2} \z(\x) &= \sum_{l=1}^{2M} b_l I^{\alpha_2 - \beta_2} h_l(\x) +\frac{\mu_2}{d} \left(\frac{\Gamma{(2)}}{\Gamma{(2 - \beta_2)}}\right) \x^{1-\beta_2} \\& -\frac{c}{d} \left(\frac{\Gamma{(2)}}{\Gamma{(2 - \beta_2)}}\right) \frac{\x^{1-\beta_2}}{\left[ \left(1-\frac{a}{b} \right) \left(1-\frac{c}{d} \right) - \mu_3 \mu_4 \eta_1 \eta_2 \left(1-\frac{a \nu_2}{b} \right) \left(1-\frac{c\nu_1}{d} \right) \right]} \\& \left\{ \mu_4 \eta_2 (1-\frac{a\nu_2}{b}) \left[\mu_3 \eta_1 \sum_{l=1}^{2M} b_l I^{\alpha_2} h_l(\nu_1) + \frac{\mu_2 \mu_3 \eta_1 \nu_1}{d} - \frac{\mu_1}{b} - \sum_{l=1}^{2M} a_l I^{\alpha_1} h_l(1) \right] + \left(1-\frac{a}{b} \right) \right. \\  & \hspace{0.5cm} \left.  \left[\mu_4 \eta_2 \sum_{l=1}^{2M} a_l I^{\alpha_1} h_l(\nu_2) + \frac{\mu_1 \mu_4 \eta_2 \nu_2}{b} - \sum_{l=1}^{2M} b_l I^{\alpha_2} h_l(1) - \frac{\mu_2}{d} \right] \right\}.
    \end{split}
\end{equation}
By substituting equations \eqref{P4_sec3eqn1}, \eqref{P4_sec3eqn8}, \eqref{P4_sec3eqn9}, \eqref{P4_sec3eqn9_1}, and \eqref{P4_sec3eqn9_2} into equation \eqref{P4_problem} and expanding at the collocation points \eqref{collocation}, we obtain system of nonlinear equations in the following form:
\begin{equation}
        \phi_c(a_1, a_2, a_3, \ldots, a_{2M}, b_1, b_2, \ldots, b_{2M}) = 0, \quad c=1,2, \ldots 4M.
\end{equation}
 We solve the above system using the Newton-Raphson method to obtain the wavelet coefficients $a_l$ and $b_l$. At last, the derived wavelet coefficients are substituted back into equations \eqref{P4_sec3eqn8} and \eqref{P4_sec3eqn9} to obtain the desired solutions.\\

\noindent \textbf{Case-II:-} Let us consider the nonlinear coupled fractional Lane-Emden equation \eqref{P4_problem} together with conditions \eqref{P4_condition_case2}.
Now, let the higher-order fractional derivative into the Haar wavelet series \cite{lepik2014haar}:
\begin{eqnarray}
    &&\label{P4_sec3eqn9_3} \D^{\alpha_1} \y(\x) = \sum_{l=1}^{2M} a_l h_l(\x),\\
    &&\label{P4_sec3eqn9_4} \D^{\alpha_2} \z(\x) = \sum_{l=1}^{2M} b_l h_l(\x).
\end{eqnarray}
Taking the fractional integral of equations \eqref{P4_sec3eqn9_3} and \eqref{P4_sec3eqn9_4} followed by Lemma \ref{P4_subsec2_lemma2_3} and conditions \eqref{P4_condition_case2}, we get,
\begin{eqnarray}
   &&\label{P4_sec3eqn10} \y(\x) =  \sum_{l=1}^{2M} a_l I^{\alpha_1} h_l (\x) + \frac{\mu_1}{a}  + \left(\x - \frac{b}{a} \right) \y'(0),\\
  &&\label{P4_sec3eqn11} \z(\x) = \sum_{l=1}^{2M} b_l I^{\alpha_2} h_l (\x) + \frac{\mu_2}{c} + \left(\x-\frac{d}{c} \right) \z'(0).
\end{eqnarray}
Now, again from conditions \eqref{P4_condition_case2}, we have  $\y(1) = \mu_3 \eta_1 \z(\nu_1)$, and $\z(1) = \mu_4 \eta_2 \y(\nu_2)$. Using \eqref{P4_sec3eqn10} and \eqref{P4_sec3eqn11} into these conditions, we determine the values of $\y'(0)$ and $\z'(0)$. Substituting these values into equations \eqref{P4_sec3eqn10} and \eqref{P4_sec3eqn11}, we obtain,
\begin{equation}\label{P4_sec3eqn14}
    \begin{split}
     \y(\x) &= \sum_{l=1}^{2M} a_l I^{\alpha_1} h_l (\x) +\frac{\mu_1}{a} +  \frac{\left(\x- \frac{b}{a} \right)}{ \left[\left(1-\frac{b}{a} \right) \left(1-\frac{d}{c} \right) -\mu_3 \mu_4 \eta_1 \eta_2 \left(\nu_1-\frac{d}{c} \right) \left(\nu_2 - \frac{b}{a} \right) \right]} \left\{ \left(1-\frac{d}{c} \right) \right.\\ & 
    \hspace{1cm} \left. \left[ \mu_3 \eta_1 \sum_{l=1}^{2M} b_l I^{\alpha_2} h_l(\nu_1)   + \frac{\mu_2 \mu_3 \eta_1}{c}  - \frac{\mu_1}{a} - \sum_{l=1}^{2M} a_l I^{\alpha_1} h_l(1) \right]  + \mu_3 \eta_1 \left(\nu_1 -\frac{d}{c} \right) \right. \\& \hspace{1cm} \left. \left[ \mu_4 \eta_2 \sum_{l=1}^{2M} a_l I^{\alpha_1} h_l(\nu_2)  + \frac{\mu_1 \mu_4 \eta_2}{a} - \frac{\mu_2}{c} - \sum_{l=1}^{2M} b_l I^{\alpha_2} h_l(1) \right] \right\},
    \end{split}
\end{equation}
\begin{equation}\label{P4_sec3eqn15}
    \begin{split}
       \z(\x) & = \sum_{l=1}^{2M} b_l I^{\alpha_2} h_l (\x) + \frac{\mu_2}{c} +  \frac{\left(\x-\frac{d}{c} \right)}{ \left[ \left(1-\frac{b}{a} \right) \left(1-\frac{d}{c} \right) -\mu_3 \mu_4 \eta_1 \eta_2 \left(\nu_1-\frac{d}{c} \right) \left(\nu_2 - \frac{b}{a} \right) \right]} \left\{ \mu_4 \eta_2 \left( \nu_2- \frac{b}{a} \right) \right.\\& \hspace{1cm} \left. \left[ \mu_3 \eta_1 \sum_{l=1}^{2M} b_l I^{\alpha_2} h_l (\nu_1)   + \frac{\mu_2 \mu_3 \eta_1}{c}  - \frac{\mu_1}{a} - \sum_{l=1}^{2M} a_l I^{\alpha_1} h_l(1) \right] + \left( 1- \frac{b}{a} \right) \right. \\& \hspace{1cm} \left. \left[ \mu_4 \eta_2 \sum_{l=1}^{2M} a_l I^{\alpha_1} h_l(\nu_2)  + \frac{\mu_1 \mu_4 \eta_2}{a} - \frac{\mu_2}{c}- \sum_{l=1}^{2M} b_l I^{\alpha_2} h_l(1) \right] \right\}.
    \end{split}
\end{equation}
Taking lower order fractional derivative of equations \eqref{P4_sec3eqn14} and \eqref{P4_sec3eqn15}, we get,
\begin{equation}\label{P4_sec3eqn16}
    \begin{split}
        \D^{\beta_1} \y(\x) &= \sum_{l=1}^{2M} a_l I^{\alpha_1 - \beta_1} h_l (\x)  + \left(\frac{\Gamma{(2)}}{\Gamma{(2-\beta_1)}}\right) \frac{\x^{(1-\beta_1)}}{\left[ \left(1-\frac{b}{a} \right) \left(1-\frac{d}{c} \right) -\mu_3 \mu_4 \eta_1 \eta_2 \left(\nu_1-\frac{d}{c} \right) \left(\nu_2 - \frac{b}{a} \right) \right]} \\& \left\{ \left(1-\frac{d}{c} \right)  \left[ \mu_3 \eta_1 \sum_{l=1}^{2M} b_l I^{\alpha_2} h_l(\nu_1)   + \frac{\mu_2 \mu_3 \eta_1}{c}  - \frac{\mu_1}{a} - \sum_{l=1}^{2M} a_l I^{\alpha_1} h_l(1) \right]  + \mu_3 \eta_1 \left(\nu_1 -\frac{d}{c} \right) \right. \\& \hspace{1cm} \left. \left[ \mu_4 \eta_2 \sum_{l=1}^{2M} a_l I^{\alpha_1} h_l(\nu_2)  + \frac{\mu_1 \mu_4 \eta_2}{a} - \frac{\mu_2}{c} - \sum_{l=1}^{2M} b_l I^{\alpha_2} h_l(1) \right] \right\},
    \end{split}
\end{equation}
\begin{equation}\label{P4_sec3eqn17}
    \begin{split}
        \D^{\beta_2} \z(\x) & = \sum_{l=1}^{2M} b_l I^{\alpha_2-\beta_2} h_l (\x) + \left(\frac{\Gamma{(2)}}{\Gamma{(2-\beta_2)}}\right) \frac{\x^{(1-\beta_2)}}{\left[ \left(1-\frac{b}{a} \right) \left(1-\frac{d}{c} \right) -\mu_3 \mu_4 \eta_1 \eta_2 \left(\nu_1-\frac{d}{c} \right) \left(\nu_2 - \frac{b}{a} \right) \right]} \\&  \left\{ \mu_4 \eta_2 \left( \nu_2- \frac{b}{a} \right) \left[ \mu_3 \eta_1 \sum_{l=1}^{2M} b_l I^{\alpha_2} h_l (\nu_1)   + \frac{\mu_2 \mu_3 \eta_1}{c}  - \frac{\mu_1}{a} - \sum_{l=1}^{2M} a_l I^{\alpha_1} h_l(1) \right] + \left( 1- \frac{b}{a} \right) \right.\\& \hspace{1cm} \left. \left[ \mu_4 \eta_2 \sum_{l=1}^{2M} a_l I^{\alpha_1} h_l(\nu_2)  + \frac{\mu_1 \mu_4 \eta_2}{a} - \frac{\mu_2}{c}- \sum_{l=1}^{2M} b_l I^{\alpha_2} h_l(1) \right] \right\}.
    \end{split}
\end{equation}
Now, substituting equations \eqref{P4_sec3eqn9_3}, \eqref{P4_sec3eqn14}, \eqref{P4_sec3eqn15}, \eqref{P4_sec3eqn16}, and \eqref{P4_sec3eqn17} into equation \eqref{P4_problem} and expanding at the collocation points \eqref{collocation}, we obtain system of nonlinear equation in the following form,
\begin{equation}
\varphi_c(a_1, a_2, a_3, \ldots, a_{2M}, b_1, b_2, \ldots, b_{2M}) = 0, \quad c= 1, 2, \ldots, 4M.
\end{equation}
We use the Newton-Raphson method to extract the solution of the above system of equations to determine the wavelet coefficients $a_l$ and $b_l$. Finally, the acquired coefficients are plugged back into equations \eqref{P4_sec3eqn14} and \eqref{P4_sec3eqn15} to obtain the required solutions.

\subsection{Algorithm}
The key steps involved in implementing the proposed method are outlined as follows:\\
\textbf{Input:} $J$, $M = 2^J$, $(\alpha_1, ~\beta_1, ~\alpha_2, ~\beta_2)$, initial solution, and iteration count.\\
\textbf{Step 1:} Compute the Haar function $h_l(\x)$ and its fractional integrals $I^{\alpha_1} h_l (\x)$, $I^{\alpha_2} h_l (\x)$, $I^{\alpha_1-\beta_1} h_l (\x)$ and $I^{\alpha_2-\beta_2} h_l (\x)$.\\
\textbf{Step 2:} Using the proposed method outlined in Section \ref{P4_collocation_method}, compute $\D^{\alpha_1} \y(\x)$, $\D^{\alpha_2} \z(\x)$, $\y(\x)$, $\z(\x)$, $\D^{\beta_1} \y(\x)$ and $\D^{\beta_2} \z(\x)$ for the given conditions (\eqref{P4_condition_case1} or \eqref{P4_condition_case2}). Then, construct the corresponding nonlinear system of equations.\\
\textbf{Step 3:} Solve the nonlinear system using the Newton-Raphson method, given the initial solution and iteration count, to determine the unknown wavelet coefficients.\\
\textbf{Step 4:} Construct the approximate solutions by substituting the obtained wavelet coefficients into the expressions for $\y(\x)$ and $\z(\x)$.\\
\textbf{Step 5:} If the total error is within an acceptable range, terminate the process and accept the obtained solution. Otherwise, increase the value of $J$ and repeat the steps until the error is minimized to the desired level of accuracy.\\
\textbf{Output:} $\y(\x)$ and $\z(\x)$.

\section{Convergence Analysis}\label{P4_convergence}
In this section, we examine the convergence of the proposed method through an error analysis. To quantify the difference between the exact and approximate solutions, we define the absolute error as follows. Let $\y_M$ and $\z_M$ represent the approximate solutions, while $\y_E$ and $\z_E$ denote the exact solutions. The corresponding absolute errors are given by:
\begin{align*}
E_{M1} = | \y_E - \y_M |, ~~E_{M2} = | \z_E - \z_M |.
\end{align*}
To measure the overall accuracy of the method, we define the total error as:
\begin{equation*}\label{totalnorm}
    \| E_M \| = \| E_{M1} \|_2 + \| E_{M2} \|_2.
\end{equation*}

\begin{theorem}\label{P4_thrm1}
Consider the system of fractional differential equation \eqref{P4_problem} subject to the conditions \eqref{P4_condition_case1}. Additionally, assume that the functions \( D^{\alpha_1 +1} \y(\x) \) and \( D^{\alpha_2 +1} \z(\x) \) are continuous on the interval \([0,1]\). Furthermore, suppose there exist positive constants \( \epsilon_1 \) and \( \epsilon_2 \) such that $|D^{\alpha_1 +1} \y(\x) | \leq \epsilon_1, ~| D^{\alpha_2 +1} \z(\x) | \leq \epsilon_2, ~\text{for all } \x \in [0,1],$ and let \( k \) be an arbitrary positive real number. Then, for \( \alpha_1 = \alpha_2 = \alpha \), the errors \( E_{M1} \) and \( E_{M2} \) in the solutions \( \y(\x) \) and \( \z(\x) \), respectively, satisfy the bounds:
\[
\| E_{M1} \|_2 \leq \frac{C_1 k}{2 \left(2^\alpha -1 \right)} \left(\frac{1}{2^{J+1}}\right)^\alpha, \quad \quad \| E_{M2} \|_2 \leq \frac{C_2 k}{2(2^\alpha -1)} \left(\frac{1}{2^{J+1}}\right)^\alpha.
\]
Consequently, the total error satisfies,
\[
\| E_M \| \leq \mathcal{O} \left( \frac{1}{2^{J+1}}\right)^\alpha,
\]
where,
\begin{equation*} \begin{split}
C_1 &= \left[ \epsilon_1^2  + \frac{\left(1- \frac{c}{d} \right)^2 \left(\frac{a^2}{3b^2} - \frac{a}{b} +1 \right) \left\{\mu_3^2 \eta_1^2 \epsilon_2^2 + \epsilon_1^2 - 2 \mu_3 \eta_1 \epsilon_1 \epsilon_2 \right\} }{ \left[ \left(1-\frac{a}{b} \right) \left(1-\frac{c}{d} \right) - \mu_3 \mu_4 \eta_1 \eta_2 \left(1-\frac{a\nu_2}{b} \right) \left(1-\frac{c\nu_1}{d} \right) \right]^2 } \right.\\ & \left.  + \frac{ \mu_3^2 \eta_1^2 \left(1-\frac{c \nu_1}{d} \right)^2 \left(\frac{a^2}{3b^2} - \frac{a}{b} +1 \right) \left\{ \mu_4^2 \eta_2^2 \epsilon_1^2 +\epsilon_2^2 - 2 \mu_4 \eta_2 \epsilon_1 \epsilon_2 \right\} }{ \left[ \left(1-\frac{a}{b} \right) \left(1-\frac{c}{d} \right) - \mu_3 \mu_4 \eta_1 \eta_2 \left(1-\frac{a\nu_2}{b} \right) \left(1-\frac{c\nu_1}{d} \right) \right]^2}  \right.\\ & \left.  + \frac{2 \mu_3 \eta_1 \left(1-\frac{c}{d} \right) \left(1-\frac{c \nu_1}{d} \right) \left(\frac{a^2}{3b^2} - \frac{a}{b} +1 \right) \left\{ \mu_3 \mu_4 \eta_1 \eta_2 \epsilon_1 \epsilon_2 - \mu_4 \eta_2 \epsilon_1^2 - \mu_3 \eta_1 \epsilon_2^2 + \epsilon_1 \epsilon_2 \right\} }{ \left[ \left(1-\frac{a}{b} \right) \left(1-\frac{c}{d} \right) - \mu_3 \mu_4 \eta_1 \eta_2 \left(1-\frac{a\nu_2}{b} \right) \left(1-\frac{c\nu_1}{d} \right) \right]^2} \right.\\ & \left. + \frac{2 \left(1-\frac{a}{2b} \right) \left\{ \mu_3 \eta_1\left( \left(1-\frac{c}{d} \right)  -  \left(1-\frac{c\nu_1}{d} \right)\right) \epsilon_1 \epsilon_2 + \epsilon_1^2 \left( \mu_3 \mu_4 \eta_1 \eta_2 \left(1-\frac{c \nu_1}{d} \right) - \left(1-\frac{c}{d} \right)\right) \right\} }{ \left[ \left(1-\frac{a}{b} \right) \left(1-\frac{c}{d} \right) - \mu_3 \mu_4 \eta_1 \eta_2 \left(1-\frac{a\nu_2}{b} \right) \left(1-\frac{c\nu_1}{d} \right) \right]} \right]^{(1/2)},
\end{split} \end{equation*}
\begin{equation*}
    \begin{split}
        C_2 &= \left[ \epsilon_2^2 + \frac{\mu_4^2 \eta_2^2 \left(1-\frac{a\nu_2}{b} \right)^2 \left(\frac{c^2}{3d^2} -\frac{c}{d}+1 \right)}{ \left[ \left(1-\frac{a}{b} \right) \left( 1-\frac{c}{d} \right) - \mu_3 \mu_4 \eta_1 \eta_2 \left(1-\frac{a\nu_2}{b} \right) \left(1-\frac{c\nu_1}{d} \right) \right]^2} \left\{ \mu_3^2 \eta_1^2 \epsilon_2^2 + \epsilon_1^2 - 2 \epsilon_1 \epsilon_2 \mu_3 \eta_1 \right\}  \right.\\& \left.  + \frac{ \left(1-\frac{a}{b} \right)^2 \left(\frac{c^2}{3d^2} -\frac{c}{d}+1 \right) }{ \left[ \left(1-\frac{a}{b} \right) \left( 1-\frac{c}{d} \right) - \mu_3 \mu_4 \eta_1 \eta_2 \left(1-\frac{a\nu_2}{b} \right) \left(1-\frac{c\nu_1}{d} \right) \right]^2} \left\{ \mu_4^2 \eta_2^2 \epsilon_1^2 - 2 \mu_4 \eta_2 \epsilon_1 \epsilon_2 \right\} \right.\\& \left.   + \frac{ 2\mu_4 \eta_2 \left(1-\frac{a}{b} \right) \left(1-\frac{a\nu_2}{b} \right) \left(\frac{c^2}{3d^2} -\frac{c}{d}+1 \right) }{ \left[ \left(1-\frac{a}{b} \right) \left( 1-\frac{c}{d} \right) - \mu_3 \mu_4 \eta_1 \eta_2 \left(1-\frac{a\nu_2}{b} \right) \left(1-\frac{c\nu_1}{d} \right) \right]^2} \left\{ \left(\mu_3 \mu_4 \eta_1 \eta_2 +1 \right) \epsilon_1 \epsilon_2 - \mu_4 \eta_2 \epsilon_1^2 - \mu_3 \eta_1 \epsilon_2^2 \right\}  \right.\\ & \left.  +   \frac{2 \left( 1- \frac{c}{2d} \right) \left\{ \left(\mu_3 \mu_4 \eta_1 \eta_2 \left(1-\frac{a\nu_2}{b} \right) \right) \epsilon_2^2 + \mu_4 \eta_2 \left( \left(1-\frac{a}{b} \right) - \left(1-\frac{a \nu_2}{b} \right) \right) \epsilon_1 \epsilon_2 \right\}}{ \left[ \left(1-\frac{a}{b} \right) \left( 1-\frac{c}{d} \right) - \mu_3 \mu_4 \eta_1 \eta_2 \left(1-\frac{a\nu_2}{b} \right) \left(1-\frac{c\nu_1}{d} \right) \right]} \right]^{\left(1/2 \right)}.
    \end{split}
\end{equation*}
  \end{theorem} 
\begin{proof} The approximate solutions of problem \eqref{P4_problem} are given as follows:
\begin{align*}
    \y_M(\x) &= \sum_{l=1}^{2M} a_l I^{\alpha_1} h_l(\x) + \frac{\mu_1 \x}{b} + \frac{ \left(1-\frac{a \x}{b} \right)}{\left[(1-\frac{a}{b}) (1-\frac{c}{d}) - \mu_3 \mu_4 \eta_1 \eta_2 (1-\frac{a\nu_2}{b}) (1-\frac{c\nu_1}{d}) \right]} \left\{ (1-\frac{c}{d}) \right. \\ & \left.
    \left[ \mu_3 \eta_1 \sum_{l=1}^{2M} b_l I^{\alpha_2} h_l(\nu_1) - \frac{\mu_1}{b} + \frac{\mu_2 \mu_3 \eta_1 \nu_1}{d} - \sum_{l=1}^{2M} a_l I^{\alpha_1} h_l(1) \right] + \mu_3 \eta_1 \left(1-\frac{c \nu_1}{d} \right) \right.\\& \left.
    \left[ \mu_4 \eta_2 \sum_{l=1}^{2M} a_l I^{\alpha_1} h_l(\nu_2) - \frac{\mu_2}{d} + \frac{\mu_1 \mu_4 \eta_2 \nu_2}{b} - \sum_{l=1}^{2M} b_l I^{\alpha_2} h_l(1) \right] \right\},
\end{align*}
\begin{align*}
    \z_M(\x) &= \sum_{l=1}^{2M} b_l I^{\alpha_2} h_l(\x) + \frac{\mu_2 \x}{d} + \frac{\left(1-\frac{c \x}{d} \right)}{ \left[ \left(1-\frac{a}{b} \right) \left(1-\frac{c}{d} \right) - \mu_3 \mu_4 \eta_1 \eta_2 \left(1-\frac{a\nu_2}{b} \right) \left(1-\frac{c\nu_1}{d} \right) \right]} \left\{ \mu_4 \eta_2 \left(1-\frac{a\nu_2}{b} \right) \right.\\& \hspace{1cm} \left. \left[ \mu_3 \eta_1 \sum_{l=1}^{2M} b_l I^{\alpha_2} h_l(\nu_1) + \frac{\mu_2 \mu_3 \eta_1 \nu_1}{d} - \frac{\mu_1}{b} - \sum_{l=1}^{2M} a_l I^{\alpha_1} (1)\right] + \left(1-\frac{a}{b} \right)  \right.\\& \hspace{1.5cm} \left. \left[\mu_4 \eta_2 \sum_{l=1}^{2M} a_l I^{\alpha_1} h_l(\nu_2) + \frac{\mu_1 \mu_4 \eta_2 \nu_2}{b} -\frac{\mu_2}{d}- \sum_{l=1}^{2M} b_l I^{\alpha_2} (1)  \right] \right\},
\end{align*}
and, the exact solutions of problem \eqref{P4_problem} are,
\begin{align*}
    \y_E(\x) &= \sum_{l=1}^{\infty} a_l I^{\alpha_1} h_l(\x) + \frac{\mu_1 \x}{b} + \frac{\left(1-\frac{a \x}{b} \right)}{\left[ \left(1-\frac{a}{b} \right) \left(1-\frac{c}{d} \right) - \mu_3 \mu_4 \eta_1 \eta_2 \left(1-\frac{a\nu_2}{b} \right) \left(1-\frac{c\nu_1}{d} \right) \right]} \left\{ \left(1-\frac{c}{d} \right) \right.\\ & \hspace{1cm} \left. 
    \left[ \mu_3 \eta_1 \sum_{l=1}^{\infty} b_l I^{\alpha_2} h_l(\nu_1) - \frac{\mu_1}{b} + \frac{\mu_2 \mu_3 \eta_1 \nu_1}{d} - \sum_{l=1}^{\infty} a_l I^{\alpha_1} h_l(1) \right] + \mu_3 \eta_1 \left(1-\frac{c \nu_1}{d} \right) \right.\\& \hspace{1.2cm} \left. 
    \left[ \mu_4 \eta_2 \sum_{l=1}^{\infty} a_l I^{\alpha_1} h_l(\nu_2) - \frac{\mu_2}{d} + \frac{\mu_1 \mu_4 \eta_2 \nu_2}{b} - \sum_{l=1}^{\infty} b_l I^{\alpha_2} h_l(1) \right] \right\},
\end{align*}
\begin{align*}
    \z_E(\x) &= \sum_{l=1}^{\infty} b_l I^{\alpha_2} h_l(\x) + \frac{\mu_2 \x}{d} + \frac{\left(1-\frac{c \x}{d} \right)}{\left[ \left(1-\frac{a}{b} \right) \left(1-\frac{c}{d} \right) - \mu_3 \mu_4 \eta_1 \eta_2 \left(1-\frac{a\nu_2}{b} \right) \left(1-\frac{c\nu_1}{d} \right) \right]} \left\{ \mu_4 \eta_2 \left(1-\frac{a\nu_2}{b} \right) \right. \\& \hspace{1cm} \left. \left[ \mu_3 \eta_1 \sum_{l=1}^{\infty} b_l I^{\alpha_2} h_l(\nu_1) + \frac{\mu_2 \mu_3 \eta_1 \nu_1}{d} - \frac{\mu_1}{b} - \sum_{l=1}^{\infty} a_l I^{\alpha_1} h_l(1)\right] + \left(1-\frac{a}{b} \right) \right. \\& \hspace{1.5cm} \left.  \left[\mu_4 \eta_2 \sum_{l=1}^{\infty} a_l I^{\alpha_1} h_l(\nu_2) + \frac{\mu_1 \mu_4 \eta_2 \nu_2}{b} -\frac{\mu_2}{d}- \sum_{l=1}^{\infty} b_l I^{\alpha_2} (1)  \right] \right\}.
\end{align*}
Therefore, the error in exact solution $\y_E(\x)$ and approximate solution $\y_M(\x)$ is,
\begin{align*}
    \| E_{M1} \|_2 &= \| \y_E(\x) - \y_M(\x) \|_2 \\
    &= \left\| \sum_{l=2M+1}^{\infty} a_l I^{\alpha_1} h_l(\x) + \frac{\left(1-\frac{a\x}{b} \right)}{\left[ \left(1-\frac{a}{b} \right) \left(1-\frac{c}{d} \right) - \mu_3 \mu_4 \eta_1 \eta_2 \left(1-\frac{a\nu_2}{b} \right) \left(1-\frac{c\nu_1}{d} \right) \right]} \left\{ \left(1-\frac{c}{d} \right) \right.\right. \\& \hspace{1cm} \left. \left. \left[ \mu_3 \eta_1 \sum_{l=2M+1}^{\infty} b_l I^{\alpha_2} h_l(\nu_1) - \sum_{l=2M+1}^{\infty} a_l I^{\alpha_1} h_l(1) \right] + \mu_3 \eta_1 (1-\frac{c\nu_1}{d}) \right. \right.\\ & \hspace{1.2cm} \left. \left. \left[ \mu_4 \eta_2 \sum_{2M+1}^{\infty} a_l I^{\alpha_1} h_l(\nu_2) - \sum_{l=2M+1}^{\infty} b_l I^{\alpha_2} h_l(1) \right] \right\} \right\|_2.
\end{align*}
Since $l = 2^j+k+1,$ it follows that:
\begin{align*}
    \| E_{M1} \|_2 &= \left\| \sum_{j= J+1}^{\infty} \sum_{k=0}^{2^j-1} a_{2^j+k+1} \mathsf{p}_{{\alpha_1}, 2^j+k+1}(\x) + \frac{ \left(1-\frac{a\x}{b} \right)}{ \left[ \left(1-\frac{a}{b} \right) \left(1-\frac{c}{d} \right) - \mu_3 \mu_4 \eta_1 \eta_2 \left(1-\frac{a\nu_2}{b}\right) \left(1-\frac{c\nu_1}{d} \right) \right]} \left\{ \left(1-\frac{c}{d} \right)   \right. \right.\\& \hspace{0.6cm} \left. \left. \left[ \mu_3 \eta_1 \sum_{j= J+1}^{\infty} \sum_{k=0}^{2^j-1} b_{2^j+k+1} \mathsf{p}_{{\alpha_2}, 2^j+k+1}(\nu_1) - \sum_{j= J+1}^{\infty} \sum_{k=0}^{2^j-1} a_{2^j+k+1} \mathsf{p}_{\alpha_1, 2^j+k+1}(1) \right] + \mu_3 \eta_1 \left(1-\frac{c\nu_1}{d}\right)  \right. \right.\\ & \hspace{1.2cm} \left. \left. \left[ \mu_4 \eta_2 \sum_{j= J+1}^{\infty} \sum_{k=0}^{2^j-1} a_{2^j+k+1} \mathsf{p}_{\alpha_1, 2^j+k+1}(\nu_2) - \sum_{j= J+1}^{\infty} \sum_{k=0}^{2^j-1} b_{2^j+k+1} \mathsf{p}_{\alpha_2, 2^j+k+1}(1) \right] \right\} \right\|_2,
\end{align*}
\begin{align*}
    \| E_{M1} \|_2^2  &= \int_0^1 \left\{ \sum_{j= J+1}^{\infty} \sum_{k=0}^{2^j-1} a_{2^j+k+1} \mathsf{p}_{\alpha_1, 2^j+k+1}(\x) + \frac{\left(1-\frac{a\x}{b} \right)}{ \left[ \left(1-\frac{a}{b} \right) \left(1-\frac{c}{d} \right) - \mu_3 \mu_4 \eta_1 \eta_2 \left(1-\frac{a\nu_2}{b} \right) \left(1-\frac{c\nu_1}{d} \right) \right]}  \right. \\& \left.  \left\{ \left(1-\frac{c}{d} \right) \left[ \mu_3 \eta_1 \sum_{j= J+1}^{\infty} \sum_{k=0}^{2^j-1} b_{2^j+k+1} \mathsf{p}_{\alpha_2, 2^j+k+1}(\nu_1) - \sum_{j= J+1}^{\infty} \sum_{k=0}^{2^j-1} a_{2^j+k+1} \mathsf{p}_{\alpha_1, 2^j+k+1}(1) \right] + \mu_3 \eta_1  \right. \right. \\ & \left. \left. \left(1-\frac{c\nu_1}{d}\right) \left[ \mu_4 \eta_2 \sum_{j= J+1}^{\infty} \sum_{k=0}^{2^j-1} a_{2^j+k+1} \mathsf{p}_{\alpha_1, 2^j+k+1}(\nu_2) - \sum_{j= J+1}^{\infty} \sum_{k=0}^{2^j-1} b_{2^j+k+1} \mathsf{p}_{\alpha_2, 2^j+k+1}(1) \right] \right\} \right\}^2 d\x, 
\end{align*}
\begin{equation} \label{P4_thm1_eqn1}
\begin{split}
     \| E_{M1} \|_2^2  &= \int_0^1 \left( \sum_{j= J+1}^{\infty} \sum_{k=0}^{2^j-1} a_{2^j+k+1} \mathsf{p}_{\alpha_1, 2^j+k+1}(\x)\right)^2 d\x + \frac{\left(1-\frac{c}{d} \right)^2}{\left[(1-\frac{a}{b}) (1-\frac{c}{d}) - \mu_3 \mu_4 \eta_1 \eta_2 (1-\frac{a\nu_2}{b}) (1-\frac{c\nu_1}{d})\right]^2} \\ & 
     \int_0^1 \left(1-\frac{a \x}{b} \right)^2 \left\{ \mu_3 \eta_1 \sum_{j= J+1}^{\infty} \sum_{k=0}^{2^j-1} b_{2^j+k+1} \mathsf{p}_{\alpha_2, 2^j+k+1}(\nu_1) - \sum_{j= J+1}^{\infty} \sum_{k=0}^{2^j-1} a_{2^j+k+1} \mathsf{p}_{\alpha_1, 2^j+k+1}(1) \right\}^2 d\x \\ & 
     + \frac{\mu_3^2 \eta_1^2 \left(1- \frac{c\nu_1}{d} \right)^2}{ \left[(1-\frac{a}{b}) (1-\frac{c}{d}) - \mu_3 \mu_4 \eta_1 \eta_2 (1-\frac{a\nu_2}{b}) (1-\frac{c\nu_1}{d}) \right]^2 } \\& 
     \int_0^1 \left(1-\frac{a \x}{b} \right)^2 \left\{ \mu_4 \eta_2  \sum_{j= J+1}^{\infty} \sum_{k=0}^{2^j-1} a_{2^j+k+1}  \mathsf{p}_{\alpha_1, 2^j+k+1}(\nu_2) - \sum_{j= J+1}^{\infty} \sum_{k=0}^{2^j-1} b_{2^j+k+1} \mathsf{p}_{\alpha_2, 2^j+k+1}(1) \right\}^2 d\x  \\ & 
     + \frac{2 \mu_3 \eta_1 \left(1-\frac{c}{d} \right) \left(1-\frac{c \nu_1}{d} \right)}{ \left[ \left(1-\frac{a}{b} \right) \left(1-\frac{c}{d} \right) - \mu_3 \mu_4 \eta_1 \eta_2 \left(1-\frac{a\nu_2}{b} \right) \left(1-\frac{c\nu_1}{d} \right) \right]^2 } \\ & \int_0^1 \left(1-\frac{a \x}{b} \right)^2 \left\{ \mu_3 \eta_1 \sum_{j= J+1}^{\infty} \sum_{k=0}^{2^j-1} b_{2^j+k+1} \mathsf{p}_{\alpha_2, 2^j+k+1}(\nu_1) - \sum_{j= J+1}^{\infty} \sum_{k=0}^{2^j-1} a_{2^j+k+1} \mathsf{p}_{\alpha_1, 2^j+k+1}(1) \right\} \\ & \left\{ \mu_4 \eta_2  \sum_{j= J+1}^{\infty} \sum_{k=0}^{2^j-1} a_{2^j+k+1}  \mathsf{p}_{\alpha_1, 2^j+k+1}(\nu_2) - \sum_{j= J+1}^{\infty} \sum_{k=0}^{2^j-1} b_{2^j+k+1} \mathsf{p}_{\alpha_2, 2^j+k+1}(1) \right\} d\x \\& + \frac{2}{\left[(1-\frac{a}{b}) (1-\frac{c}{d}) - \mu_3 \mu_4 \eta_1 \eta_2 (1-\frac{a\nu_2}{b}) (1-\frac{c\nu_1}{d}) \right]} \int_0^1 \left(1-\frac{a \x}{b} \right) \left( \sum_{j= J+1}^{\infty} \sum_{k=0}^{2^j-1} a_{2^j+k+1} \mathsf{p}_{\alpha_1, 2^j+k+1}(\x)\right)   \\ & 
     \left\{ \left(1-\frac{c}{d} \right) \left\{ \mu_3 \eta_1 \sum_{j= J+1}^{\infty} \sum_{k=0}^{2^j-1} b_{2^j+k+1} \mathsf{p}_{\alpha_2, 2^j+k+1}(\nu_1) - \sum_{j= J+1}^{\infty} \sum_{k=0}^{2^j-1} a_{2^j+k+1} \mathsf{p}_{\alpha_1, 2^j+k+1}(1) \right\} \right. \\& \left. + \mu_3 \eta_1 \left(1-\frac{c \nu_1}{d} \right) \left\{ \mu_4 \eta_2  \sum_{j= J+1}^{\infty} \sum_{k=0}^{2^j-1} a_{2^j+k+1}  \mathsf{p}_{\alpha_1, 2^j+k+1}(\nu_2) - \sum_{j= J+1}^{\infty} \sum_{k=0}^{2^j-1} b_{2^j+k+1} \mathsf{p}_{\alpha_2, 2^j+k+1}(1) \right\} \right\} d\x.
\end{split}
\end{equation}
Now, using Lemma \ref{P4_lemma4} and the orthogonal property of Haar function \eqref{P4_haar_ortho}, we obtain the bounds of $a_{2^j +k+1}$, $b_{2^j +k+1}$, $\mathsf{p}_{{\alpha_1}, 2^j +k+1} (\x)$ and $\mathsf{p}_{{\alpha_2}, 2^j +k+1} (\x)$ same as in \cite{lok2024, narendra2025}. Thus we have,
\begin{eqnarray}
    && \label{P4_haar_coff} a_{2^j +k+1} \leq \epsilon_1 \left( \frac{1}{2^{j+1}}\right), \quad b_{2^j +k+1} \leq \epsilon_2 \left( \frac{1}{2^{j+1}}\right),\\
    && \label{FIvalue} |\mathsf{p}_{{\alpha_1}, 2^j +k+1} (\x)| \leq \omega_1 \left( \frac{1}{2^{j+1}}\right)^{\alpha_1}, \quad |\mathsf{p}_{{\alpha_2}, 2^j +k+1} (\x)| \leq \omega_2 \left( \frac{1}{2^{j+1}}\right)^{\alpha_2}, \quad \forall \x \in [0,1],~\omega_1, \omega_2 \in \mathbb{R^+}.
\end{eqnarray}
Substituting equations \eqref{P4_haar_coff} and \eqref{FIvalue} into equation \eqref{P4_thm1_eqn1},  we arrive at,
\begin{equation}\label{P4_thm1_eqn2}
    \begin{split}
        \| E_{M1} \|_2^2 &\leq \omega_1^2 \epsilon_1^2 \left( \frac{1}{2(2^{\alpha_1}-1)} \right)^2  \left(\frac{1}{2^{J+1}}\right)^{2\alpha_1} \left\{ 1 + \frac{2 \left(1-\frac{a}{2b} \right) \left[\mu_3 \mu_4 \eta_1 \eta_2 (1-\frac{c\nu_1}{d}) - ( 1-\frac{c}{d}) \right]}{\left[(1-\frac{a}{b}) (1-\frac{c}{d}) - \mu_3 \mu_4 \eta_1 \eta_2 (1-\frac{a\nu_2}{b}) (1-\frac{c\nu_1}{d}) \right]} \right.\\ & + \left. \frac{ \left(\frac{a^2}{3b^2} - \frac{a}{b} +1 \right)}{ \left[(1-\frac{a}{b}) (1-\frac{c}{d}) - \mu_3 \mu_4 \eta_1 \eta_2 (1-\frac{a\nu_2}{b}) (1-\frac{c\nu_1}{d}) \right]^2} \left[ \left(1-\frac{c}{d} \right)^2 + \mu_3^2 \mu_4^2 \eta_1^2 \eta_2^2 \left(1-\frac{c\nu_1}{d} \right)^2  \right. \right.\\& \left. 
        \left.  - 2 \mu_3\mu_4 \eta_1 \eta_2 \left(1-\frac{c \nu_1}{d} \right) \left(1-\frac{c}{d} \right) \right] \right\} + \omega_2^2 \epsilon_2^2 \left( \frac{1}{2(2^{\alpha_2}-1)} \right)^2  \left(\frac{1}{2^{J+1}}\right)^{2\alpha_2} \\&  \frac{\left(\frac{a^2}{3b^2} - \frac{a}{b} +1 \right)}{\left[ \left(1-\frac{a}{b}\right) \left(1-\frac{c}{d} \right) - \mu_3 \mu_4 \eta_1 \eta_2 \left(1-\frac{a\nu_2}{b} \right) \left(1-\frac{c\nu_1}{d} \right) \right]^2}  \left\{ \mu_3^2 \eta_1^2 \left[ \left(1-\frac{c}{d} \right)^2 + \left(1-\frac{c \nu_1}{d} \right)^2 \right]  \right.\\& \left.  -2 \mu_3^2 \eta_1^2 \left(1-\frac{c\nu_1}{d}\right) \left(1-\frac{c}{d} \right) \right\} + \omega_1  \omega_2 \epsilon_1 \epsilon_2 \left( \frac{1}{2}\right)^2 \left( \frac{1}{(2^{\alpha_1}-1)} \right) \left( \frac{1}{(2^{\alpha_1}-1)} \right)  \left(\frac{1}{2^{J+1}}\right)^{\alpha_1 + \alpha_2} \\& \left\{ \frac{2\mu_3 \eta_1 \left(1-\frac{a}{2b} \right) \left[ \left(1-\frac{c}{d} \right) - \left(1-\frac{c \nu_1}{d} \right) \right]}{ \left[ \left(1-\frac{a}{b} \right) \left(1-\frac{c}{d} \right) - \mu_3 \mu_4 \eta_1 \eta_2 \left(1-\frac{a\nu_2}{b} \right) \left(1-\frac{c\nu_1}{d} \right) \right]} + \frac{2 \mu_3 \eta_1 \left(\frac{a^2}{3b^2} - \frac{a}{b} +1 \right)}{\left[ \left(1-\frac{a}{b} \right) \left(1-\frac{c}{d} \right) - \mu_3 \mu_4 \eta_1 \eta_2 \left(1-\frac{a\nu_2}{b} \right) \left(1-\frac{c\nu_1}{d} \right) \right]^2} \right.\\
        & \left. \left[ \left(1-\frac{c \nu_1}{d} \right) \left(1-\frac{c}{d} \right) \left\{\mu_3 \mu_4 \eta_1 \eta_2 +1 \right\} - \mu_3 \mu_4 \eta_1 \eta_2 \left(1-\frac{c \nu_1}{d} \right)^2 - \left(1-\frac{c}{d} \right)^2 \right]   \right\}.
    \end{split}
\end{equation}
For $\alpha_1 = \alpha_2 =\alpha$ and $\omega_1 = \omega_2 = k,$ equation \eqref{P4_thm1_eqn2} becomes,
\begin{equation*}
    \begin{split}
        \| E_{M1} \|_2^2 &\leq k^2 \left( \frac{1}{2(2^\alpha-1)}\right)^2 \left(\frac{1}{2^{J+1}} \right)^{2\alpha} \left[ \epsilon_1^2  + \frac{\left(1- \frac{c}{d} \right)^2 \left(\frac{a^2}{3b^2} - \frac{a}{b} +1 \right) \left\{\mu_3^2 \eta_1^2 \epsilon_2^2 + \epsilon_1^2 - 2 \mu_3 \eta_1 \epsilon_1 \epsilon_2 \right\} }{ \left[(1-\frac{a}{b}) (1-\frac{c}{d}) - \mu_3 \mu_4 \eta_1 \eta_2 (1-\frac{a\nu_2}{b}) (1-\frac{c\nu_1}{d}) \right]^2 } \right. \\ & \left. + \frac{ \mu_3^2 \eta_1^2 \left(1-\frac{c \nu_1}{d} \right)^2 \left(\frac{a^2}{3b^2} - \frac{a}{b} +1 \right) \left\{ \mu_4^2 \eta_2^2 \epsilon_1^2 +\epsilon_2^2 - 2 \mu_4 \eta_2 \epsilon_1 \epsilon_2 \right\} }{\left[(1-\frac{a}{b}) (1-\frac{c}{d}) - \mu_3 \mu_4 \eta_1 \eta_2 (1-\frac{a\nu_2}{b}) (1-\frac{c\nu_1}{d}) \right]^2} \right. \\ & \left. +\frac{2 \mu_3 \eta_1 \left(1-\frac{c}{d} \right) \left(1-\frac{c \nu_1}{d} \right) \left(\frac{a^2}{3b^2} - \frac{a}{b} +1 \right) \left\{ (\mu_3 \mu_4 \eta_1 \eta_2 +1) \epsilon_1 \epsilon_2 - \mu_4 \eta_2 \epsilon_1^2 - \mu_3 \eta_1 \epsilon_2^2 \right\} }{\left[ \left(1-\frac{a}{b} \right) \left(1-\frac{c}{d} \right) - \mu_3 \mu_4 \eta_1 \eta_2 \left(1-\frac{a\nu_2}{b} \right) \left(1-\frac{c\nu_1}{d} \right) \right]^2} \right.\\ & \left. + \frac{2 \left(1-\frac{a}{2b} \right) \left\{ \left( \left(1-\frac{c}{d} \right) \mu_3 \eta_1 - \mu_3 \eta_1 \left(1-\frac{c\nu_1}{d} \right)\right) \epsilon_1 \epsilon_2 + \epsilon_1^2 \left( \mu_3 \mu_4 \eta_1 \eta_2 \left(1-\frac{c \nu_1}{d} \right) - \left(1-\frac{c}{d} \right)\right) \right\} }{ \left[ \left(1-\frac{a}{b} \right) \left(1-\frac{c}{d} \right) - \mu_3 \mu_4 \eta_1 \eta_2 \left(1-\frac{a\nu_2}{b} \right) \left(1-\frac{c\nu_1}{d} \right) \right]} \right], \\& 
        \leq C_1^2 k^2 \left( \frac{1}{2(2^\alpha-1)}\right)^2 \left(\frac{1}{2^{J+1}} \right)^{2\alpha}. 
    \end{split}
\end{equation*}
So,
\begin{equation}
    \| E_{M1} \|_2 \leq \frac{C_1 k}{2(2^\alpha -1)} \left(\frac{1}{2^{J+1}}\right)^\alpha.
\end{equation}
Similarly,
\begin{equation}
    \| E_{M2} \|_2 \leq \frac{C_2 k}{2(2^\alpha -1)} \left(\frac{1}{2^{J+1}}\right)^\alpha,
\end{equation}
where,
\begin{equation*}
 \begin{split}
        C_2 &= \left[ \epsilon_2^2 + \frac{\mu_4^2 \eta_2^2 \left(1-\frac{a\nu_2}{b} \right)^2 \left(\frac{c^2}{3d^2} -\frac{c}{d}+1 \right) \left\{ \mu_3^2 \eta_1^2 \epsilon_2^2 + \epsilon_1^2 - 2 \epsilon_1 \epsilon_2 \mu_3 \eta_1 \right\}}{ \left[ \left(1-\frac{a}{b} \right) \left( 1-\frac{c}{d} \right) - \mu_3 \mu_4 \eta_1 \eta_2 \left(1-\frac{a\nu_2}{b} \right) \left(1-\frac{c\nu_1}{d} \right) \right]^2}   \right. \\& \left. + \frac{\left(1-\frac{a}{b} \right)^2 \left(\frac{c^2}{3d^2} -\frac{c}{d}+1 \right) }{\left[\left(1-\frac{a}{b}\right) \left( 1-\frac{c}{d} \right) - \mu_3 \mu_4 \eta_1 \eta_2 \left(1-\frac{a\nu_2}{b} \right) \left(1-\frac{c\nu_1}{d}\right) \right]^2} \left\{ \mu_4^2 \eta_2^2 \epsilon_1^2 - 2 \mu_4 \eta_2 \epsilon_1 \epsilon_2 \right\} \right.\\& \left. + \frac{ 2\mu_4 \eta_2 \left(1-\frac{a}{b} \right) \left(1-\frac{a\nu_2}{b} \right) \left(\frac{c^2}{3d^2} -\frac{c}{d}+1 \right) }{\left[\left(1-\frac{a}{b}\right) \left( 1-\frac{c}{d}\right) - \mu_3 \mu_4 \eta_1 \eta_2 \left(1-\frac{a\nu_2}{b}\right) \left(1-\frac{c\nu_1}{d}\right) \right]^2} \left\{ (\mu_3 \mu_4 \eta_1 \eta_2 +1) \epsilon_1 \epsilon_2 - \mu_4 \eta_2 \epsilon_1^2 - \mu_3 \eta_1 \epsilon_2^2 \right\} \right. \\ & \left. +   \frac{2 \left( 1- \frac{c}{2d} \right) \left\{ \left(\mu_3 \mu_4 \eta_1 \eta_2 \left(1-\frac{a\nu_2}{b}\right) \right) \epsilon_2^2 + \mu_4 \eta_2 \left( \left(1-\frac{a}{b} \right) - \left(1-\frac{a \nu_2}{b} \right) \right) \epsilon_1 \epsilon_2 \right\}}{ \left[ \left(1-\frac{a}{b}\right) \left( 1-\frac{c}{d} \right) - \mu_3 \mu_4 \eta_1 \eta_2 \left(1-\frac{a\nu_2}{b} \right) \left(1-\frac{c\nu_1}{d} \right) \right]}   \right]^{\left(1/2 \right)}.
    \end{split}
\end{equation*}
Thus the total error is,
\begin{equation}
    \begin{split}
        \| E_M \| &= \| E_{M1} \|_2 + \| E_{M2} \|_2, \\
        & \leq (C_1 + C_2) \left(\frac{k}{2(2^\alpha-1)}\right) \left(\frac{1}{2^{J+1}}\right)^\alpha, \\
        & \leq \mathcal{O} \left( \frac{1}{2^{J+1}}\right)^\alpha.
    \end{split}
\end{equation}
As $M=2^J$ and $J$ is the maximal level of resolution, the above equation ensures convergence.
\end{proof} 
\begin{theorem}\label{P4_thrm2}
    Consider the system of fractional differential equations \eqref{P4_problem} subject to the conditions \eqref{P4_condition_case2}. Suppose that the functions $D^{\alpha_1+1} \y(\x)$ and $D^{\alpha_2+1} \z(\x) $ are continuous functions on the interval $[0,1].$ Additionally, assume that there exist positive constants $\delta_1$ and $\delta_2$ such that $ |D^{\alpha_1+1} \y(\x) | \leq \delta_1, ~| D^{\alpha_2+1} \z(\x) | \leq \delta_2$ for all $\x \in [0,1]$ and $k$ be any positive real value. Then, for $\alpha_1 = \alpha_2 =\alpha ,$ the error $E_{M1}$ and $E_{M2}$ in the solutions $\y(\x)$ and $\z(\x)$ respectively satisfy the bounds,
\begin{equation*}
    \| E_{M1} \|_2 \leq  \left( \frac{d_1 k}{2(2^\alpha -1)} \right) \left( \frac{1}{2^{J+1}} \right)^\alpha,
\end{equation*}
\begin{equation*}
    \| E_{M2} \|_2 \leq  \left( \frac{d_2 k}{2(2^\alpha -1)} \right) \left( \frac{1}{2^{J+1}} \right)^\alpha.
\end{equation*}
Consequently, the total error satisfies,
\[
\| E_M \| \leq \mathcal{O} \left( \frac{1}{2^{J+1}}\right)^\alpha,
\]
where,
\begin{equation*}
    \begin{split}
        d_1 &= \left[ \delta_1^2 + \frac{\left(1-\frac{d}{c} \right)^2 \left(\frac{b^2}{a^2} -\frac{b}{a} + \frac{1}{3} \right)}{\left[ \left(1-\frac{b}{a} \right) \left(1-\frac{d}{c} \right) -\mu_3 \mu_4 \eta_1 \eta_2 \left(\nu_1-\frac{d}{c} \right) \left(\nu_2 - \frac{b}{a} \right) \right]^2 } \left\{ \mu_3^2 \eta_1^2 \delta_2^2 +\delta_1^2 - 2 \delta_1 \delta_2 \right\} \right. \\& \left.  + \frac{\mu_3^2 \eta_1^2 \left(\frac{b^2}{a^2} -\frac{b}{a} + \frac{1}{3} \right) }{ \left[ \left(1-\frac{b}{a}\right) \left(1-\frac{d}{c} \right) -\mu_3 \mu_4 \eta_1 \eta_2 \left(\nu_1-\frac{d}{c} \right) \left(\nu_2 - \frac{b}{a} \right) \right]^2 } \left\{ \mu_4^2 \eta_2^2 \delta_1^2 + \delta_2^2 -2 \mu_4 \eta_2 \delta_1 \delta_2 \right\} \right. \\& \left.  +
    \frac{2 \mu_3 \eta_1 \left(1-\frac{d}{c} \right) \left(\nu_1 -\frac{d}{c} \right) \left(\frac{b^2}{a^2} -\frac{b}{a} + \frac{1}{3} \right)}{\left[ \left(1-\frac{b}{a} \right) \left(1-\frac{d}{c} \right) -\mu_3 \mu_4 \eta_1 \eta_2 \left(\nu_1-\frac{d}{c} \right) \left(\nu_2 - \frac{b}{a}\right) \right]^2} \left\{ \left(\mu_3 \mu_4 \eta_1 \eta_2 +1 \right) \delta_1 \delta_2 - \mu_4 \eta_2 \delta_1^2 - \mu_3 \eta_1 \delta_2^2 \right\}  \right. \\&  \left. + \frac{2 \left(\frac{1}{2} -\frac{b}{a} \right)}{ \left[ \left(1-\frac{b}{a} \right) \left(1-\frac{d}{c} \right) -\mu_3 \mu_4 \eta_1 \eta_2 \left(\nu_1-\frac{d}{c} \right) \left(\nu_2 - \frac{b}{a} \right) \right]} \left\{ \delta_1 \delta_2 \left\{ \left(1-\frac{d}{c} \right) \mu_3 \eta_1 - \mu_3 \eta_1 \left(\nu_1 -\frac{d}{c} \right) \right\} \right. \right.\\&  \left. \left. + \delta_1^2 \{ \mu_3 \mu_4 \eta_1 \eta_2 (\nu_1-\frac{d}{c}) - (1- \frac{d}{c}) \} \right\} \right]^{\left(1/2 \right)}.
    \end{split}
\end{equation*}
\begin{equation*}
    \begin{split}
        d_2 &=\left[  \delta_2^2 + \frac{\mu_4^2 \eta_2^2 \left(\nu_2 -\frac{b}{a} \right)^2 \left(\frac{d^2}{c^2} -\frac{d}{c} + \frac{1}{3} \right)}{ \left[ \left(1-\frac{b}{a} \right) \left(1-\frac{d}{c} \right) -\mu_3 \mu_4 \eta_1 \eta_2 \left(\nu_1-\frac{d}{c} \right) \left(\nu_2 - \frac{b}{a} \right) \right]^2} \left\{ \mu_3^2 \eta_2 \delta_2^2 + \delta_1^2 - 2 \mu_3 \eta_1 \delta_1 \delta_2 \right\} \right. \\ &
        \hspace{0.5cm} \left.  + \frac{ \left(1-\frac{b}{a} \right)^2 \left(\frac{d^2}{c^2} -\frac{d}{c} + \frac{1}{3} \right) }{ \left[ \left(1-\frac{b}{a} \right) \left(1-\frac{d}{c} \right) -\mu_3 \mu_4 \eta_1 \eta_2 \left(\nu_1-\frac{d}{c} \right) \left(\nu_2 - \frac{b}{a} \right) \right]^2}  \left\{ \mu_4^2 \eta_2^2 \delta_1^2 + \delta_2^2 - 2 \mu_4 \eta_2 \delta_1 \delta_2 \right\}  \right. \\ &
          \hspace{1cm} \left.  + \frac{2 \mu_4 \eta_2 \left(1-\frac{b}{a} \right) \left(\nu_2 -\frac{b}{a} \right) \left(\frac{d^2}{c^2} -\frac{d}{c} + \frac{1}{3} \right)  }{ \left[ \left(1-\frac{b}{a} \right) \left(1-\frac{d}{c} \right) -\mu_3 \mu_4 \eta_1 \eta_2 \left(\nu_1-\frac{d}{c} \right) \left(\nu_2 - \frac{b}{a} \right) \right]^2} \left\{ \left(\mu_3 \mu_4 \eta_1 \eta_2 +1 \right) \delta_1 \delta_2 - \mu_4 \eta_2 \delta_1^2 - \mu_3 \eta_1 \delta_2^2 \right\} \right. \\ &
          \hspace{1cm} \left.  + \frac{2 \left(\frac{1}{2}- \frac{d}{c} \right) \left\{ \mu_4 \eta_2 \left(\nu_2 -\frac{b}{a} \right) \left[\mu_3 \eta_1 \delta_2^2 - \delta_1 \delta_2 \right] + \left(1-\frac{b}{a} \right) \left[\mu_4 \eta_2 \delta_1 \delta_2 - \delta_2^2 \right] \right\} }{ \left[ \left(1-\frac{b}{a} \right) \left(1-\frac{d}{c} \right) -\mu_3 \mu_4 \eta_1 \eta_2 \left(\nu_1-\frac{d}{c} \right) \left(\nu_2 - \frac{b}{a}\right) \right]} \right]^{\left(1/2 \right)}.
    \end{split}
\end{equation*}
\end{theorem}
\begin{proof} The proof is similar, as done in above Theorem \ref{P4_thrm1}.
\end{proof}
\section{Numerical Experiments}\label{P4_Examples}
In this section, we present numerical experiments to validate the proposed method by considering five test problems of the form \eqref{P4_problem}. Additionally, we introduce stability analysis and maximum residual error, which will be utilized throughout this section. All the experiments are conducted using \textit{Mathematica 11.3} on a system with a 64-bit Intel Core i7 CPU and 16GB of RAM.\\

\noindent \textbf{Stability analysis:-} To ensure the stability of the proposed algorithm, it is essential to evaluate the condition number of the resulting nonlinear system of algebraic equations. These equations arise from coupled fractional differential equations after applying the proposed method. The algorithm is considered stable if the condition number of the coefficient matrix remains bounded \cite{SUN_2021106732}. In this context, the coefficient matrix is constructed using Haar functions, the unknown variables correspond to wavelet coefficients, and the right-hand side consists of known values. A commonly used criterion for numerical stability \cite{narendra2025} states that the method remains stable if the norm of the inverse of the coefficient matrix exists and is bounded by a fixed constant. This guarantees that small variations in input do not cause significant errors in the solution.\\

\noindent \textbf{Maximum residual error:} For Lane-Emden-type equations, where exact solutions are unavailable, the maximum residual error is defined as:
$$E = \max_{\x_i \in [0, 1]} r(x_i), \quad r(x_i) = \sqrt{r_1(\x_i)^2 + r_2(\x_i)^2},$$
where $r(x_i)$ denotes the total residual error at a given point $x_i$. The component residual errors $r_1(\x)$ and $r_2(\x)$ for the system are defined as follows:
\begin{align*}
    r_1(\x) &= \left|\D^{\alpha_1} \y^J(\x) + \frac{k_1}{\x^{\gamma_1}} \D^{\beta_1} \y^J(\x) -f_1(\x,\y^J,\z^J) \right|, \\
    r_2(\x) &= \left|\D^{\alpha_2} \z^J(\x) + \frac{k_2}{\x^{\gamma_2}} \D^{\beta_2} \z^J(\x) -f_2(\x,\y^J,\z^J) \right|,
\end{align*}
where, $\y^J(\x)$ and $\z^J(\x)$ denote the numerical solutions at the $J^{th}$ resolution level. 

\subsection{Fractional Initial Value Problem(FrIVP)}\label{P4_IVP}
We consider the following FrIVP:
\begin{equation}
\begin{aligned}
& \D^{\alpha_1} \y(\x) +\frac{1}{\x} \D^{\beta_1} \y(\x) = \z^3 \left(\y^2+1 \right),\\
& \D^{\alpha_2} \z(\x) + \frac{3}{\x} \D^{\beta_2} (\z (\x))= -\z^5 \left(\y^2+3 \right),\\
& \y(0)=1, ~\y'(0)=0, \quad \z(0)=1, ~\z'(0)=0.
\end{aligned}
\end{equation}
\begin{itemize}
    \item[\textbf{(a)} ]We present our findings systematically through figures and tables. Figure \ref{P4_IVP_fig} illustrates the Haar solution and residual error, arranged column-wise for different values $\alpha_1, ~\beta_1, ~\alpha_2, ~\beta_2.$ The first column displays the Haar solutions $\y(\x), \z(\x)$ along with the residual errors in $r_1(\x)$ and $r_2(\x)$ for varying values of $J$ at $\alpha_1=1.58, ~\beta_1=0.58,~ \alpha_2=1.59, ~ \beta_2=0.59$ respectively.
    \item[\textbf{(b)}]The second column presents the same components for $\alpha_1=1.7,~ \beta_1=0.7,~ \alpha_2=1.71, ~\beta_2=0.71,$ while the third column follows with solutions and residual errors at $\alpha_1=1.85,~ \beta_1=0.85,~ \alpha_2=1.86,~ \beta_2=0.86.$ The fourth column showcases results for $\alpha_1=1.98,~ \beta_1=0.98,~ \alpha_2=1.99,~ \beta_2=0.99.$
    \item[\textbf{(c)} ]Finally, the last column presents the Haar solutions $\y(\x), \z(\x)$ alongside the absolute error in $\y(\x)$ and $\z(\x)$ for various values of $J$ at $\alpha_1=2,\; \beta_1=1,\; \alpha_2=2,\; \beta_2=1.$ At $\alpha_1 = \alpha_2 =2,~ \beta_1 = \beta_2=1$, the proposed problem reduces to the classical coupled Lane-Emden equations considered by Narendra et al. \cite{kumar2023hybrid}.
    \item[\textbf{(d)} ]In Table \ref{P4_table1_IVP}, we present the total residual error of problem \ref{P4_IVP} for  $J=3, \; 4,\; 5$ at  $\alpha_1=1.58,\; 1.7, ~\alpha_2=1.59, \;1.71 , ~ \beta_1=0.58,\; 0.7,~ \beta_2=0.59,~ 0.71.$ Similarly, Table \ref{P4_table2_IVP} displays the total residual error for  $J=3,\; 4,\; 5$ at  $\alpha_1=1.85,\; 1.98, ~\alpha_2=1.86, \; 1.99 , ~\beta_1=0.85, \; 0.98, ~\beta_2=0.86, \; 0.99.$
    \item[\textbf{(e)} ]The tables and plots demonstrate that the residual error decreases as $J$ increases for the given values of parameters $\alpha_1,\; \beta_1, \; \alpha_2,\; \beta_2,$ A similar trend is observed when $J$ remains fixed while varying $\alpha_1,\; \beta_1, \; \alpha_2, \; \beta_2,$ which validates the accuracy of the proposed method. The residual error shown in Figure \ref{P4_IVP_fig} serves as a measure of the precision of our method. These results demonstrate the effectiveness of the fractional Haar wavelet collocation method in tackling the problem.
    \item[\textbf{(f)} ]We also observe that the condition number of the coefficient matrix remains bounded across all considered cases. Additionally, the norm of the inverse of this matrix is bounded as well. Hence, we conclude that the proposed method demonstrates numerical stability.
    \item[\textbf{(g)} ]Despite changes in the initial solution, the final solution remains unchanged across all cases of $(\alpha_1, ~\beta_1, ~\alpha_2, ~\beta_2)$. This consistency confirms the stability of the proposed method.
    \end{itemize}
\begin{figure}[H]
  \centering
  \includegraphics[width=1.0\textwidth]{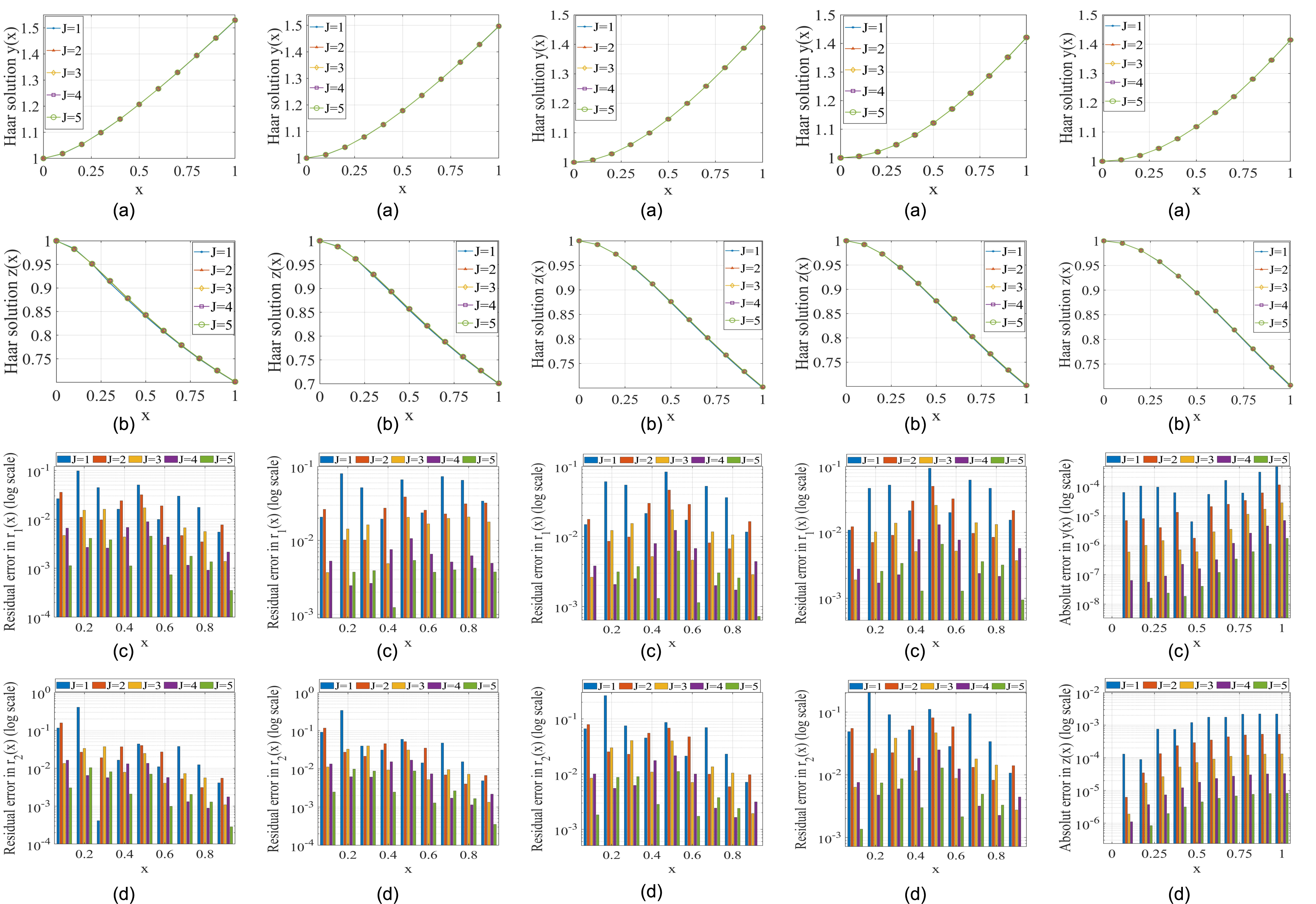} 
  \caption{Solution and error plots of example \ref{P4_IVP}, with varying parameters $\alpha_1$, $\beta_1$, $\alpha_2$, and $\beta_2$ arranged column-wise (from left to right): (a) $\alpha_1 = 1.58$, $\beta_1 = 0.58$, $\alpha_2 = 1.59$, $\beta_2 = 0.59$; (b) $\alpha_1 = 1.7$, $\beta_1 = 0.7$, $\alpha_2 = 1.71$, $\beta_2 = 0.71$; (c) $\alpha_1 = 1.85$, $\beta_1 = 0.85$, $\alpha_2 = 1.86$, $\beta_2 = 0.86$; (d) $\alpha_1 = 1.98$, $\beta_1 = 0.98$, $\alpha_2 = 1.99$, $\beta_2 = 0.99$; and (e) $\alpha_1 = 2$, $\beta_1 = 1$, $\alpha_2 = 2$, $\beta_2 = 1$.}
  \label{P4_IVP_fig}
\end{figure}
\begin{table}[H]
  \begin{center}
    \caption{Total residual error of problem \ref{P4_IVP} at varying values of $\alpha_1,~  \beta_1,~\alpha_2~\text{and}~ \beta_2$ for various value of $J.$}
    \label{P4_table1_IVP}
    \begin{tabular}{|c|c|c|c|c|c|c|}
    \hline
     \multirow{2}{*}{x} & \multicolumn{3}{c|}{$(\alpha_1,~\beta_1, ~\alpha_2,~\beta_2) = (1.58,~ 0.58,~1.59,~0.59)$} &  \multicolumn{3}{c|}{$(\alpha_1,~\beta_1,~ \alpha_2,~\beta_2) =
(1.7,~ 0.7,~ 1.71,~0.71)$}\\
\cline{2-7}
&$J=3$ &$J=4$  &$J=5$ &$J=3$ &$J=4$ & $J=5$\\
\hline
0.1 & 0.014275 & 0.017655 & 0.003259 & 0.011851 & 0.014496 & 0.002643 \\
\hline
0.2 & 0.036819 & 0.007014 & 0.011174 & 0.035923 & 0.006708 & 0.010599 \\
\hline
0.3 & 0.040516 & 0.006223 & 0.008958 & 0.043111 & 0.006658 & 0.009608  \\
\hline
0.4 & 0.009009 & 0.01477 & 0.002377 & 0.010576 & 0.017144 & 0.002765  \\
\hline
0.5 & 0.03011 & 0.0162 & 0.008383 & 0.037501 & 0.020049 & 0.010349 \\
\hline
0.6 & 0.00502 & 0.007184 & 0.001235 & 0.006438 & 0.009211 & 0.001581  \\
\hline
0.7 & 0.00988 & 0.001749 & 0.0027 & 0.013008 & 0.002292 & 0.003531  \\
\hline
0.8 & 0.007947 & 0.001272 & 0.001869 & 0.01056 & 0.001693 & 0.002488 \\
\hline
0.9 & 0.001757 & 0.002774 & 0.000453 & 0.002385 & 0.003745 & 0.000613  \\
\hline
\textbf{$E$} & \textbf{0.040516} & \textbf{0.017655} & \textbf{0.011174} & \textbf{0.043111} & \textbf{0.020049} & \textbf{0.010599} \\
\hline
    \end{tabular}
  \end{center}
\end{table}
\begin{table}[H]
  \begin{center}
    \caption{Total residual error of problem \ref{P4_IVP} at varying values of $\alpha_1,~  \beta_1,~\alpha_2~\text{and}~ \beta_2$ for various value of $J.$}
    \label{P4_table2_IVP}
    \begin{tabular}{|c|c|c|c|c|c|c|}
    \hline
     \multirow{2}{*}{x} & \multicolumn{3}{c|}{$(\alpha_1,~\beta_1,~ \alpha_2,~\beta_2) = (1.85,~ 0.85,~ 1.86,~0.86)$} &  \multicolumn{3}{c|}{$(\alpha_1,~\beta_1,~ \alpha_2,~\beta_2) = (1.98,~ 0.98, ~1.99,~0.99)$}\\
\cline{2-7}
&$J=3$ &$J=4$  &$J=5$ &$J=3$ &$J=4$ & $J=5$\\
\hline
0.1 & 0.008864263 & 0.010780913 & 0.001941089 & 0.006640035 & 0.008068393 & 0.001439572\\
 \hline
0.2 & 0.032358591 & 0.005923749 & 0.009281468 & 0.028024954 & 0.005060426 & 0.007881939\\
 \hline
0.3 & 0.043279505 & 0.006730453 & 0.009744434 & 0.040923062 & 0.006402001 & 0.009295701\\
 \hline
0.4 & 0.012065981 & 0.019315675 & 0.003124156 & 0.01270759 & 0.02015193 & 0.003268217\\
 \hline
0.5 & 0.04676703 & 0.024823853 & 0.012775002 & 0.053577266 & 0.028270966 & 0.014510215\\
 \hline
0.6 & 0.008456794 & 0.012105698 & 0.002072213 & 0.010221575 & 0.014652558 & 0.002502738\\
 \hline
0.7 & 0.017902254 & 0.00313974 & 0.004827635 & 0.022770754 & 0.003978818 & 0.006107856\\
 \hline
0.8 & 0.014882517 & 0.002386259 & 0.003508788 & 0.019482779 & 0.003124543 & 0.004595591\\
 \hline
0.9 & 0.003458445 & 0.005406854 & 0.00088531 & 0.004644217 & 0.007244466 & 0.001186479\\
 \hline
\textbf{$E$} & \textbf{0.04676703} & \textbf{0.024823853} & \textbf{0.012775002} & \textbf{0.053577266} & \textbf{0.028270966} & \textbf{0.014510215}\\
 \hline
 \end{tabular}
  \end{center}
\end{table}
\subsection{Fractional Boundary Value Problem (FrBVP)}\label{P4_BVP}
%Let us consider the system of fractional differential equation \eqref{P4_problem} with $k_1=5, \gamma_1 = 1, f_1(\x, \y,\z)= - 8 \exp{(\y-1)} + 2 \exp{\Big(-(\frac{\z-1}{2})\Big)}$ and $k_2= 3, \gamma_2= 1, f_2(\x, \y, \z)= 8 \exp{(-(\z-1))} +\exp{\Big(\frac{\y-1}{2}\Big)},$ subject to boundary conditions $\y'(0)=0, \y(1)=1-2\log{2}, \z'(0) =0, \z(1) = 1+2 \log{2}.$
We consider the following FrBVP as:
\begin{equation}
\begin{aligned}
& \D^{\alpha_1} \y(\x) +\frac{5}{\x} \D^{\beta_1} \y(\x) = - 8 \exp{\left(\y-1 \right)} + 2 \exp{\left(- \left(\frac{\z-1}{2} \right)\right)},\\
& \D^{\alpha_2} \z(\x) + \frac{3}{\x} \D^{\beta_2} (\z (\x))= 8 \exp{\left(- \left(\z-1 \right) \right)} +\exp{\left(\frac{\y-1}{2}\right)},\\
& \y'(0)=0, ~\y(1)=1-2\log{2},\quad \z'(0) =0,~ \z(1) = 1+2 \log{2}.
\end{aligned}
\end{equation}

\begin{itemize}
    \item[\textbf{(a)}] Figure \ref{P4_BVP_fig} illustrates the Haar solution and residual error, organized column-wise for different values $\alpha_1,\; \beta_1, \; \alpha_2, \; \text{and} \; \beta_2.$ In the first column, we display the Haar solutions $\y(\x), \z(\x)$ along with the residual errors $r_1(\x)$ and $r_2(\x)$ for varying values of $J$ at $\alpha_1=1.56,~ \beta_1=0.56,~ \alpha_2=1.57,~ \beta_2=0.57$ respectively. 
    \item[\textbf{(b)}] The second column showcases how these components evolve when the parameters are adjusted for $\alpha_1=1.72,\; \beta_1=0.72,\; \alpha_2=1.73,\; \beta_2=0.73.$ As we move to the third column, we observe the solutions and residual errors at $\alpha_1=1.83,\; \beta_1=0.83,\; \alpha_2=1.84,\; \beta_2=0.84.$ The fourth column showcases results for $\alpha_1=1.99,\; \beta_1=0.98,\; \alpha_2=1.98,\; \beta_2=0.99.$
    \item[\textbf{(c)}] Finally, the last column presents the Haar solutions $\y(\x),~ \z(\x)$ alongside the absolute error in $\y(\x)$ and $\z(\x)$ these for various values of $J$ at $\alpha_1= \alpha_2=2,\; \beta_1= \beta_2=1.$ For $\alpha_1=\alpha_2=2,~ \beta_1= \beta_2=1,$ the proposed problem reduces to the classical case of BVP discussed in \cite{verma2021haar}. 
    \item[\textbf{(d)}] To further quantify these observations, Table \ref{P4_table1_BVP} presents the computed total residual error for the problem \ref{P4_BVP} for  $J=3,\; 4,\; 5$ at  $\alpha_1=1.56,\; 1.72,~ \alpha_2=1.57, \; 1.73, ~\beta_1=0.56, 0.72, ~\beta_2=0.57, 0.73.$
    \item[\textbf{(e)}] Similarly, Table \ref{P4_table2_BVP} extends this analysis by presenting total residual error for  $J=3,\; 4,\; 5$ at  $\alpha_1=1.83,~ 1.99, ~\alpha_2=1.84,~1.98,~ \beta_1=0.83,~ 0.98, ~\beta_2=0.84,~ 0.99.$ For various values of $\alpha_1,\; \beta_1, \; \alpha_2, \;\beta_2,$ the tables and figures clearly show that the residual error decreases as $J$ grows, and for $J$ staying constant while $\alpha_1,\; \beta_1,\; \alpha_2,\; \beta_2$ varies, a similar pattern is seen, which demonstrates the accuracy of our approach.
    \item[\textbf{(f)} ]Our observations indicate that the condition number of the coefficient matrix stays bounded under all tested scenarios. Furthermore, the norm of the inverse of this matrix remains bounded. Consequently, the proposed method exhibits numerical stability.
    \item[\textbf{(g)} ]The final solution remains invariant to changes in the initial solution for all parameter cases $(\alpha_1, ~\beta_1, ~\alpha_2, ~\beta_2)$, demonstrating the stability of the proposed method.
\end{itemize}
\begin{figure}[H]
  \centering
  \includegraphics[width=1.0\textwidth]{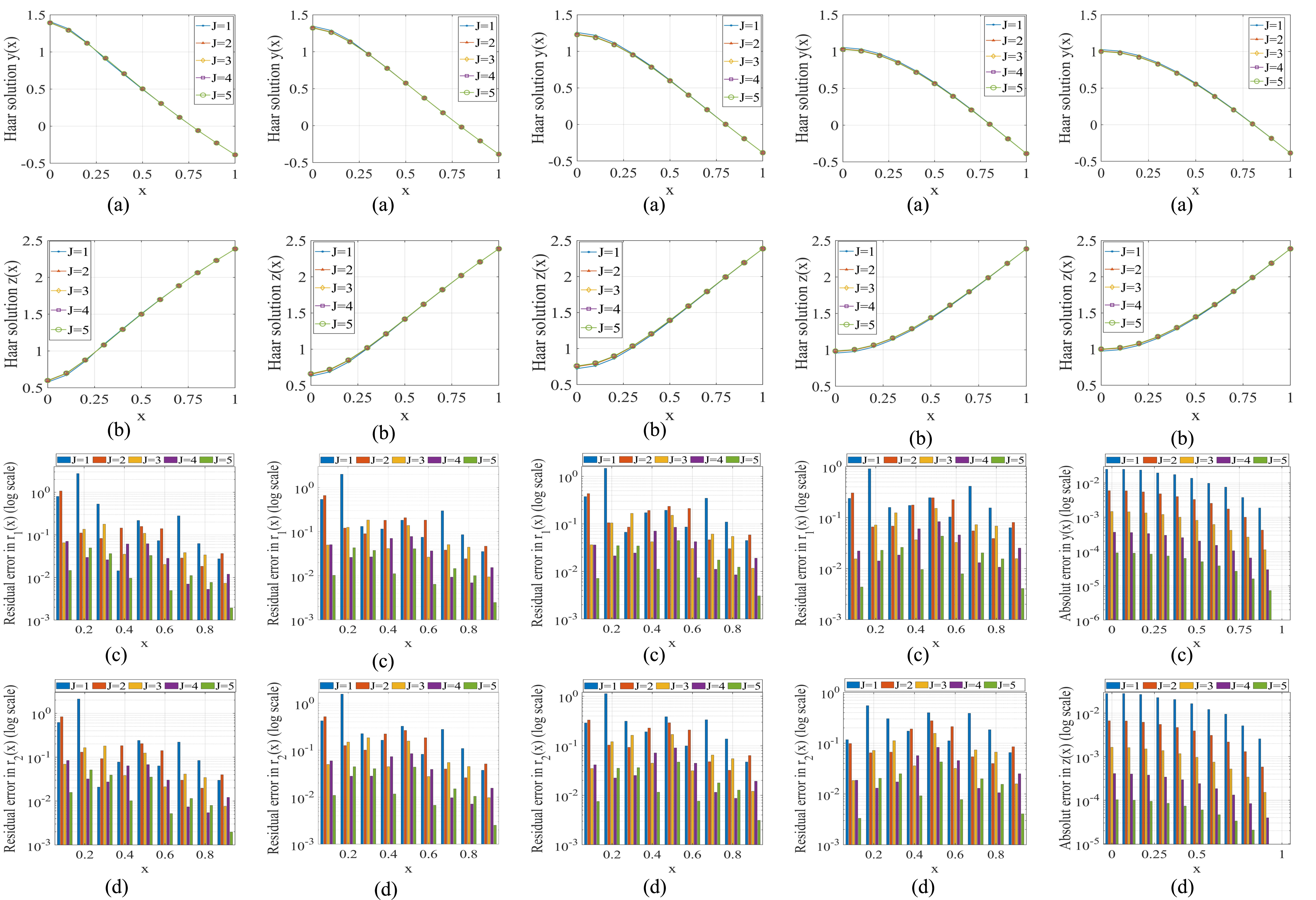} 
  \caption{Solution and error plots of example \ref{P4_BVP}, with varying parameters $\alpha_1$, $\beta_1$, $\alpha_2$, and $\beta_2$ arranged column-wise (from left to right): (a) $\alpha_1 = 1.56$, $\beta_1 = 0.56$, $\alpha_2 = 1.57$, $\beta_2 = 0.57$; (b) $\alpha_1 = 1.72$, $\beta_1 = 0.72$, $\alpha_2 = 1.73$, $\beta_2 = 0.73$; (c) $\alpha_1 = 1.83$, $\beta_1 = 0.83$, $\alpha_2 = 1.84$, $\beta_2 = 0.84$; (d) $\alpha_1 = 1.99$, $\beta_1 = 0.98$, $\alpha_2 = 1.98$, $\beta_2 = 0.99$; and (e) $\alpha_1 = 2$, $\beta_1 = 1$, $\alpha_2 = 2$, $\beta_2 = 1$.} 
  \label{P4_BVP_fig}
\end{figure}
\begin{table}[H]
  \begin{center}
    \caption{Total residual error of problem \ref{P4_BVP} at various values of $\alpha_1,~  \beta_1,~\alpha_2~\text{and}~ \beta_2$ for various value of $J.$}
    \label{P4_table1_BVP}
    \begin{tabular}{|c|c|c|c|c|c|c|}
    \hline
     \multirow{2}{*}{x} & \multicolumn{3}{c|}{$(\alpha_1,\;\beta_1,\; \alpha_2,\;\beta_2) = (1.56,\; 0.56,\; 1.57,\; 0.57)$} &  \multicolumn{3}{c|}{$(\alpha_1,\; \beta_1,\; \alpha_2,\; \beta_2) = (1.72,\; 0.72,\; 1.73,\; 0.73)$}\\
\cline{2-7}
&$J=3$ &$J=4$  &$J=5$ &$J=3$ &$J=4$ & $J=5$\\
\hline
0.1 & 0.095166285 & 	0.10985955 & 	0.021471377 & 	0.070498547	 & 0.078189757	 & 0.014967547\\
\hline
0.2	 & 0.213666169	 & 0.043545515 & 	0.071255702 & 	0.194707716 & 	0.038333606 & 	0.061910197\\
\hline
0.3 & 	0.253668907 & 	0.037772527 & 	0.053429333 & 	0.25925813 & 	0.038851241 & 	0.055135781\\
\hline
0.4 & 	0.052158857 & 	0.088335812 & 	0.014027915 & 	0.060996731 & 	0.101456171 & 	0.016160048\\
\hline
0.5 & 	0.165401489 & 	0.091864603 & 	0.048201974 & 	0.208807915 & 	0.114805432 & 	0.0600067\\
\hline
0.6 & 	0.029378282 & 	0.041222003 & 	0.007173726 & 	0.038215848 & 	0.053589371 & 	0.009299162\\
\hline
0.7 & 	0.056321995 & 	0.010212269 & 	0.015931865 & 	0.074620523 & 	0.013467348 & 	0.020971629\\
\hline
0.8 & 	0.048256772 & 	0.007600155 & 	0.01107051 & 	0.063305242 & 	0.009962144 & 	0.014498798\\
\hline
0.9 & 	0.010589804 & 	0.017058372 & 	0.002765188 & 	0.013556134 & 	0.021839492 & 	0.003536987\\
\hline
\textbf{$E$} & 	\textbf{0.253668907} & 	\textbf{0.10985955} & 	\textbf{0.071255702} & 	\textbf{0.25925813} & 	\textbf{0.114805432}	 & \textbf{0.061910197}\\
 \hline
 \end{tabular}
  \end{center}
\end{table}
\begin{table}[H]
  \begin{center}
    \caption{Total residual error of problem \ref{P4_BVP} at varying values of $\alpha_1,~  \beta_1,~\alpha_2~\text{and}~ \beta_2$ for various value of $J.$}
    \label{P4_table2_BVP}
    \begin{tabular}{|c|c|c|c|c|c|c|}
    \hline
     \multirow{2}{*}{x} & \multicolumn{3}{c|}{$(\alpha_1,\; \beta_1,\; \alpha_2,\; \beta_2) = (1.83,~ 0.83,~ 1.84,~ 0.84)$} &  \multicolumn{3}{c|}{$(\alpha_1,~ \beta_1,~ \alpha_2,~ \beta_2) = (1.99,~ 0.98,~ 1.98,~ 0.99)$}\\
\cline{2-7}
&$J=3$ &$J=4$  &$J=5$ &$J=3$ &$J=4$ & $J=5$\\
\hline
0.1 & 0.050311181 &	0.054944919 &	0.010402485 &	0.02412831 &	0.029017974 &	0.005490234\\
\hline
0.2 & 0.160152291 &	0.030928128 &	0.049577862 &	0.101558767 &	0.019277344 &	0.030691359\\
\hline
0.3 & 0.233192417 &	0.035122382 &	0.049976443 &	0.166824304 &	0.025320481 &	0.036177022\\
\hline
0.4 & 0.061345757 &	0.100767373 &	0.016092423 &	0.051302867 &	0.0830128 &	0.013314\\
\hline
0.5 & 0.227914316 &	0.124284553 &	0.064743762 &	0.21816228 &	0.117597412 &	0.060947534\\
\hline
0.6 & 0.043713233 &	0.061346376 &	0.010615007 &	0.045695843 &	0.06440292 &	0.011090827\\
\hline
0.7 & 0.08937878 &	0.016042336 &	0.024925141 &	0.103559247 &	0.018403497 &	0.02846982\\
\hline
0.8 & 0.077366021 &	0.012177103 &	0.017726251 &	0.095314179 &	0.015048309 &	0.021950448\\
\hline
0.9 & 0.016838413 &	0.027052393 &	0.004381459 &	0.022312941 &	0.035555332 &	0.00576867\\
\hline
\textbf{$E$} & 	\textbf{0.233192417} & \textbf{0.124284553} &	\textbf{0.064743762} &	\textbf{0.21816228} &	\textbf{0.117597412}	 &  \textbf{0.060947534}\\
 \hline
 \end{tabular}
  \end{center}
\end{table}

\subsection{Fractional Four-Point Boundary Value Problem}\label{P4_4ptBVP}
%Here we consider coupled fractional differential equation \eqref{P4_problem} with $k_1=\frac{1}{2}, \gamma_1 =1, f_1 (\x, \y, \z) =  -\Big\{ \frac{99}{35} \x -\frac{1}{2} + \z \Big(\x^2 -\frac{66}{35} \x^3 + \frac{1089}{1225} \x^4 \Big) - \y^2 \z \Big\}$ and $k_2= \frac{1}{2}, \gamma_2=1, f_2(\x, \y, \z)= -\Big\{ -\frac{24}{35} \x + \frac{64}{1225} \x^5 - \frac{2112}{42875} \x^6 - \y \z^2 \Big\},$ subject to $\y(0)=0, \y(1) = \z(1/2), \z(0)=0, \z(1) = \y(1/3).$
We consider the following fractional four-point boundary value problem as:
\begin{equation}
\begin{aligned}
& \D^{\alpha_1} \y(\x) +\frac{1/2}{\x} \D^{\beta_1} \y(\x) = -\left\{ \frac{99}{35} \x -\frac{1}{2} + \z \left(\x^2 -\frac{66}{35} \x^3 + \frac{1089}{1225} \x^4 \right) - \y^2 \z \right\},\\
& \D^{\alpha_2} \z(\x) + \frac{1/2}{\x} \D^{\beta_2} \z(\x)= -\left\{ -\frac{24}{35} \x + \frac{64}{1225} \x^5 - \frac{2112}{42875} \x^6 - \y \z^2 \right\},\\
& \y(0)=0, ~ \y(1) = \z(1/2), \quad \z(0)=0, ~ \z(1) = \y(1/3).
\end{aligned}
\end{equation}
\begin{itemize}
    \item[\textbf{(a)}]  Figure \ref{P4_4ptBVP_fig} illustrates the Haar solution and residual error, arranged column-wise for different parameter values $\alpha_1,\; \beta_1,\; \alpha_2,\; \; \text{and} \; \beta_2.$  The first column displays the Haar solutions $\y(\x)$ and $\z(\x)$ with the residual errors $r_1(\x)$ and $r_2(\x)$, for varying values of parameter $J$. This column focuses on the parameter set of $\alpha_1=1.56$,\; $\beta_1=0.58$,\; $\alpha_2=1.58$, and $\beta_2=0.56$. 
    \item[\textbf{(b)}] Moving to the second column, we present the corresponding results for the adjusted parameters $\alpha_1=1.69$, $\beta_1=0.71$, $\alpha_2=1.71$, and $\beta_2=0.70$. The third column continues this thorough examination with solutions and residual errors for $\alpha_1=1.84$, $\beta_1=0.85$, $\alpha_2=1.85$, and $\beta_2=0.86$. The fourth column shows the same components for the values $\alpha_1=1.98$, $\beta_1=0.99$, $\alpha_2=1.99$, and $\beta_2=0.98$.
    \item[\textbf{(c)}] Finally, the last column encapsulates the Haar solutions $\y(\x)$ and $\z(\x)$ with the absolute errors in both solutions for various values of $J$ at $\alpha_1=2$, $\beta_1=1$, $\alpha_2=2$ and $\beta_2=1$.  Here, when $\alpha_1 =\alpha_2 = 2, ~\beta_1= \beta_2 =1$, the proposed problem reduces to the classical form of the four-point coupled Lane-Emden equations discussed in \cite{verma2021haar}.
    \item[\textbf{(d)}] In Table \ref{P4_table1_4ptBVP}, we diligently present the total residual error associated with problem \ref{P4_4ptBVP} for $J=3,~ 4,~ 5,$ at $\alpha_1=1.56,\; 1.69, ~\alpha_2=1.58,~1.71,~ \beta_1=0.58, 0.71, ~ \beta_2=0.56,~ 0.7.$ Table \ref{P4_table2_4ptBVP} provides a detailed breakdown of the total residual error for the same values of $J=3,~ 4, ~5$ at  $\alpha_1=1.84,~ 1.98, ~\alpha_2=1.85,~1.99, ~\beta_1=0.85, ~0.99, ~\beta_2=0.86,~ 0.98.$
    \item[\textbf{(e)}] The tables and figures clearly show that as $ J $ increases, the residual error consistently decreases, even when $\alpha_1,$ $\beta_1,$ $\alpha_2$, and $\beta_2 $ vary while keeping $J$ constant, a similar trend occurs, highlighting the robustness and accuracy of the method.
    \item[\textbf{(f)}]The analysis reveals that, in all tested cases, the condition number of the coefficient matrix does not exceed certain bounds. Similarly, the inverse matrix norm remains limited, indicating strong numerical stability for the proposed technique.
    \item[\textbf{(g)}] Variations in the initial solution do not affect the final solution for any given $(\alpha_1, ~\beta_1, ~\alpha_2, ~\beta_2)$, indicating the robustness and stability of the proposed method.
 \end{itemize}
\begin{figure}[H]
  \centering
  \includegraphics[width=1.0\textwidth]{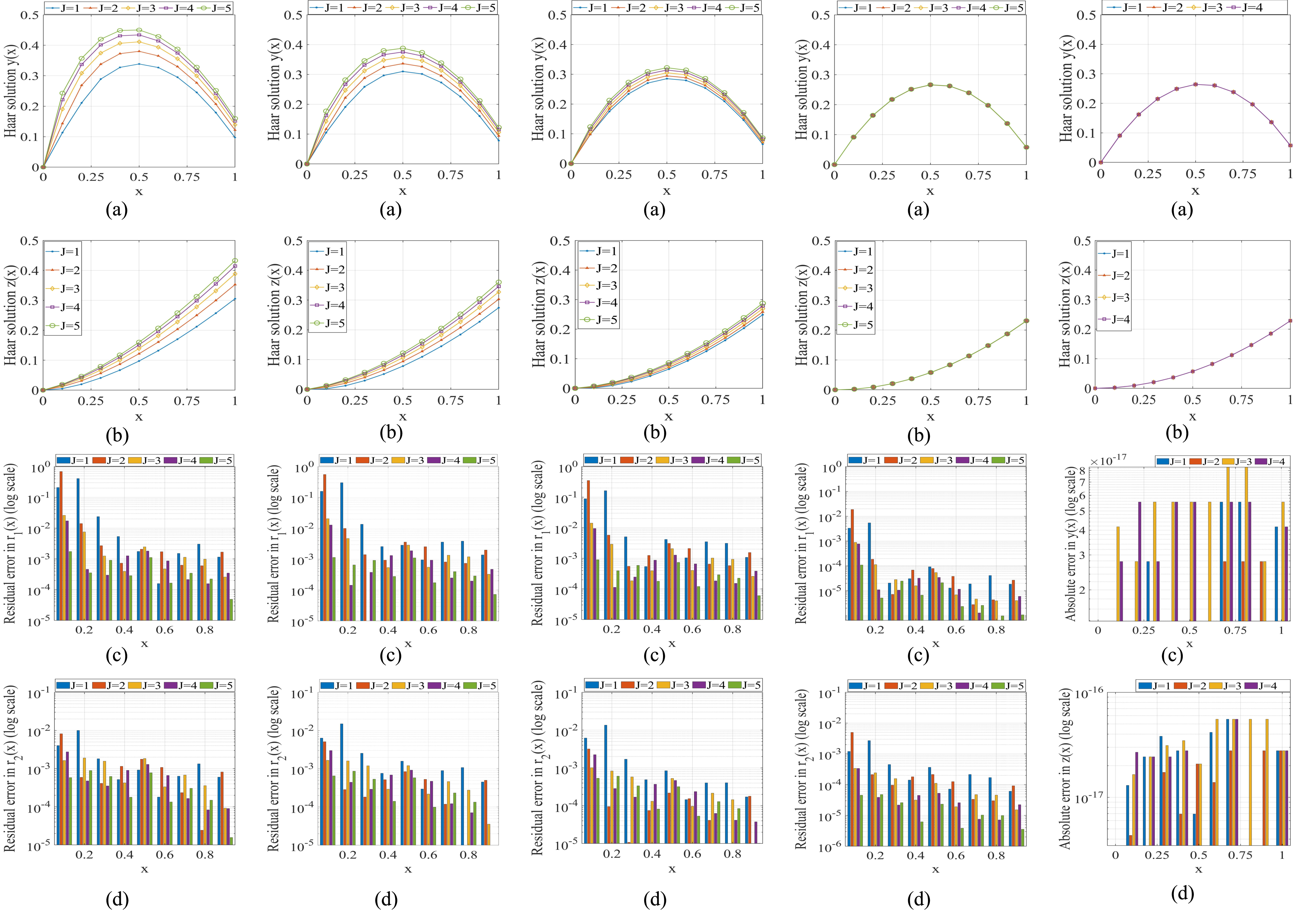} 
  \caption{Solution and error plots of example \ref{P4_4ptBVP}, with varying parameters $\alpha_1$, $\beta_1$, $\alpha_2$, and $\beta_2$ arranged column-wise (from left to right): (a) $\alpha_1 = 1.56$, $\beta_1 = 0.58$, $\alpha_2 = 1.58$, $\beta_2 = 0.56$; (b) $\alpha_1 = 1.69$, $\beta_1 = 0.71$, $\alpha_2 = 1.71$, $\beta_2 = 0.7$; (c) $\alpha_1 = 1.84$, $\beta_1 = 0.85$, $\alpha_2 = 1.85$, $\beta_2 = 0.86$; (d) $\alpha_1 = 1.98$, $\beta_1 = 0.99$, $\alpha_2 = 1.99$, $\beta_2 = 0.98$; and (e) $\alpha_1 = 2$, $\beta_1 = 1$, $\alpha_2 = 2$, $\beta_2 = 1$.} 
  \label{P4_4ptBVP_fig}
\end{figure}

\begin{table}[H]
  \begin{center}
    \caption{Total residual error of problem \ref{P4_4ptBVP} at varying values of $\alpha_1,~  \beta_1,~\alpha_2~\text{and}~ \beta_2$ for various value of $J.$}
    \label{P4_table1_4ptBVP}
    \begin{tabular}{|c|c|c|c|c|c|c|}
    \hline
     \multirow{2}{*}{x} & \multicolumn{3}{c|}{$(\alpha_1,\; \beta_1,\; \alpha_2,\; \beta_2) = (1.56,\; 0.58,\; 1.58,\; 0.56)$} &  \multicolumn{3}{c|}{$(\alpha_1,\; \beta_1,\; \alpha_2,\; \beta_2) = (1.69,\; 0.71,\; 1.71,\; 0.7)$}\\
\cline{2-7}
&$J=3$ &$J=4$  &$J=5$ &$J=3$ &$J=4$ & $J=5$\\
\hline
0.1 &	0.026235012 &	0.017523825 &	0.001837204 &	0.020560805 &	0.013195631 &	0.001319613\\
\hline
0.2 &	0.007761867 &	0.000662974 &	0.000958718 &	0.004963069 &	0.000467463 &	0.001082679\\
\hline
0.3 &	0.002005921 &	0.00046189 &	0.001101721 &	0.001194831 &	0.000476804 &	0.00106586\\
\hline
0.4 &	0.000578982 &	0.00155299 &	0.00033711 &	0.000610977 &	0.001488191 &	0.000311411\\
\hline
0.5 &	0.003076936 &	0.002218475 &	0.001357433 &	0.003114194 &	0.00209396 &	0.001249682\\
\hline
0.6 &	0.000580388 &	0.001088187 &	0.000212686 &	0.000592025 &	0.001047831 &	0.00020001\\
\hline
0.7 &	0.001321078 &	0.000271561 &	0.000462196 &	0.001427184 &	0.000277028 &	0.00045972\\
\hline
0.8 &	0.001047245 &	0.000178229 &	0.000268582 &	0.001231008 &	0.00020884 &	0.000316828\\
\hline
0.9 &	0.000275617 &	0.000359438 &	5.13499E-05 &	0.000328971 &	0.000471681 &	7.2731E-05\\
\hline
\textbf{$E$}	 & \textbf{0.026235012} &	\textbf{0.017523825} &	\textbf{0.001837204} &	\textbf{0.020560805} &	\textbf{0.013195631} &	\textbf{0.001319613}\\
\hline
 \end{tabular}
  \end{center}
\end{table}

\begin{table}[H]
  \begin{center}
    \caption{Total residual error of problem \ref{P4_4ptBVP} at varying values of $\alpha_1,~  \beta_1,~\alpha_2~\text{and}~ \beta_2$ for various value of $J.$}
    \label{P4_table2_4ptBVP}
    \begin{tabular}{|c|c|c|c|c|c|c|}
    \hline
     \multirow{2}{*}{x} & \multicolumn{3}{c|}{$(\alpha_1,\; \beta_1,\; \alpha_2,\; \beta_2) = (1.84,\; 0.85,\; 1.85,\; 0.86)$} &  \multicolumn{3}{c|}{$(\alpha_1,\; \beta_1,\; \alpha_2,\; \beta_2) = (1.98,\; 0.99,\; 1.99,\; 0.98)$}\\
\cline{2-7}
&$J=3$ &$J=4$  &$J=5$ &$J=3$ &$J=4$ & $J=5$\\
\hline
0.1 & 0.01428091 &	0.009766053 &	0.001053608 &	0.000947617 &	0.000840957 &	0.000118814\\
\hline
0.2 & 0.003015328 &	0.00030787 &	0.000723941 &	0.00026836 &	4.02341E-05 &	4.8115E-05\\
\hline
0.3 & 0.000606883 &	0.000299479 &	0.000680265 &	0.000160331 &	2.43722E-05 &	3.60821E-05\\
\hline
0.4 & 0.000417036 &	0.000949359 &	0.000197701 &	3.55083E-05 &	5.55019E-05 &	9.13333E-06\\
\hline
0.5 & 0.002150104 &	0.001364235 &	0.000803253 &	0.00012394 &	6.32511E-05 &	3.15239E-05\\
\hline
0.6 & 0.000416085 &	0.000703256 &	0.000131988 &	2.03894E-05 &	2.8432E-05 &	4.5303E-06\\
\hline
0.7 & 0.001047845 &	0.000194978 &	0.000317806 &	4.76727E-05 &	7.67584E-06 &	1.06179E-05\\
\hline
0.8 & 0.000943306 &	0.000158549 &	0.000240382 &	4.59488E-05 &	7.21583E-06 &	9.97976E-06\\
\hline
0.9 & 0.000261158 &	0.000386741 &	6.15837E-05 &	1.58286E-05 &	2.32056E-05 &	3.73574E-06\\
\hline
\textbf{$E$}	 & \textbf{0.01428091}	 & \textbf{0.009766053}	 & \textbf{0.001053608}	 & \textbf{0.000947617}	 & \textbf{0.000840957}	 & \textbf{0.000118814}\\
\hline
 \end{tabular}
  \end{center}
\end{table}
 \subsection{Fractional Catalytic Diffusion Problem}\label{P4_cat_dif}
 %Consider the differential equation \eqref{P4_problem} with $k_1= 2, \gamma_1= 1, f_1(\x, \y, \z)= - \y^2(\x) - \frac{2}{5}  (\y(\x) \z(\x)),$ and $k_2=2, \gamma_2=1 f_2(\x, \y, \z)=- \frac{1}{2} \y^2(\x) - \y(\x) \z(\x), $ subject to $\y'(0)=0, \z'(0)=0, \y(1)=1, \z(1)=2.$ This system represents the fractional version of catalytic diffusion equations.
We consider the following fractional catalytic diffusion problem as, 
\begin{equation}
\begin{aligned}
& \D^{\alpha_1} \y(\x) + \frac{2}{x} \D^{\beta_1} \y(\x) = - \y^2(\x) - \frac{2}{5} \y(\x) \z(\x),\\
& \D^{\alpha_2} \z(\x) + \frac{2}{x} \D^{\beta_2} \z(\x) = - \frac{1}{2} \y^2(\x) - \y(\x) \z(\x),\\
& \y'(0)=0,~ \z'(0)=0, \quad \y(1)=1,~ \z(1)=2.
\end{aligned}
\end{equation} 
\begin{itemize}
    \item[\textbf{(a)}] Figure \ref{P4_cat_dif_fig} visually represents our findings, with each column corresponding to a specific set of parameter values $\alpha_1,\; \beta_1,\; \alpha_2, \; \text{and} \; \beta_2.$ In the first column, we present the Haar solutions $\y(\x),\; \z(\x)$ with the residual errors in $r_1(\x)$ and $r_2(\x)$ for various values of $J$ at $\alpha_1=1.61,\; \beta_1=0.62,\; \alpha_2=1.62,\; \beta_2=0.63$ respectively.
    \item[\textbf{(b)}] The second column illustrates the same components for $\alpha_1=1.74, \;\beta_1=0.74,\; \alpha_2=1.75,\; \beta_2=0.75.$ Moving to the third column, we observe the solutions and residual errors at $\alpha_1=1.85,~ \beta_1=0.84,~ \alpha_2=1.84,~ \beta_2=0.86.$ The fourth column highlights results for $\alpha_1=1.99,~ \beta_1=0.999,~ \alpha_2=1.99,~ \beta_2=0.999.$
    \item[\textbf{(c)}] Finally, the last column presents the Haar solutions $\y(\x),~ \z(\x)$ alongside the residual error in $\y(\x)$ and $\z(\x)$ these for various values of $J$ at $\alpha_1=\alpha_2=2,~ \beta_1=\beta_2=1.$ At $\alpha_1 = \alpha_2= 2,~ \beta_1= \beta_2 =1$, the proposed problem reduces to the classical case of the catalytic diffusion equations discussed in \cite{singh2020solving}.
    \item[\textbf{(d)}]  To further validate these results we present the computed total residual error in table \ref{P4_table1_cat_dif} of problem \ref{P4_cat_dif} for $J=3,\; 4,\; 5$ at  $\alpha_1=1.61, ~1.74, ~\alpha_2=1.62,~1.75, ~\beta_1=0.62,~ 0.74,~ \beta_2=0.63,~ 0.75.$ 
    \item[\textbf{(e)}] Similarly, table \ref{P4_table2_cat_dif} provides the total residual error for $J=3,~ 4,~ 5$ at  $\alpha_1=1.85,~ 1.99, ~\alpha_2=1.84,~1.99,~ \beta_1=0.84,~ 0.999,~ \beta_2=0.86,~ 0.999.$ The tables and figures indicate that as $J$ increases with $\alpha_1,~ \beta_1,~ \alpha_2, ~ \text{and} ~ \beta_2$ fixed, the residual error steadily decreases. Similar trends occurs when $J$ is fixed and $\alpha_1,~ \beta_1,~ \alpha_2, \; \text{and} \; \beta_2$ vary, and when $\alpha_1 = \alpha_2= 2,~ \beta_1= \beta_2 =1$ the residual error becomes negligible.
    \item[\textbf{(f)}] The condition number of the coefficient matrix is consistently bounded throughout the analyzed cases. Additionally, we noted that the norm of its inverse is similarly bounded, demonstrating the numerical stability of the proposed method.
    \item[\textbf{(g)}] For any given set of parameters $(\alpha_1, ~\beta_1, ~\alpha_2, ~\beta_2)$, modifying the initial solution does not impact the final solution, validating the stability of the proposed method.
   \end{itemize}
\begin{figure}[H]
  \centering
  \includegraphics[width=1.0\textwidth]{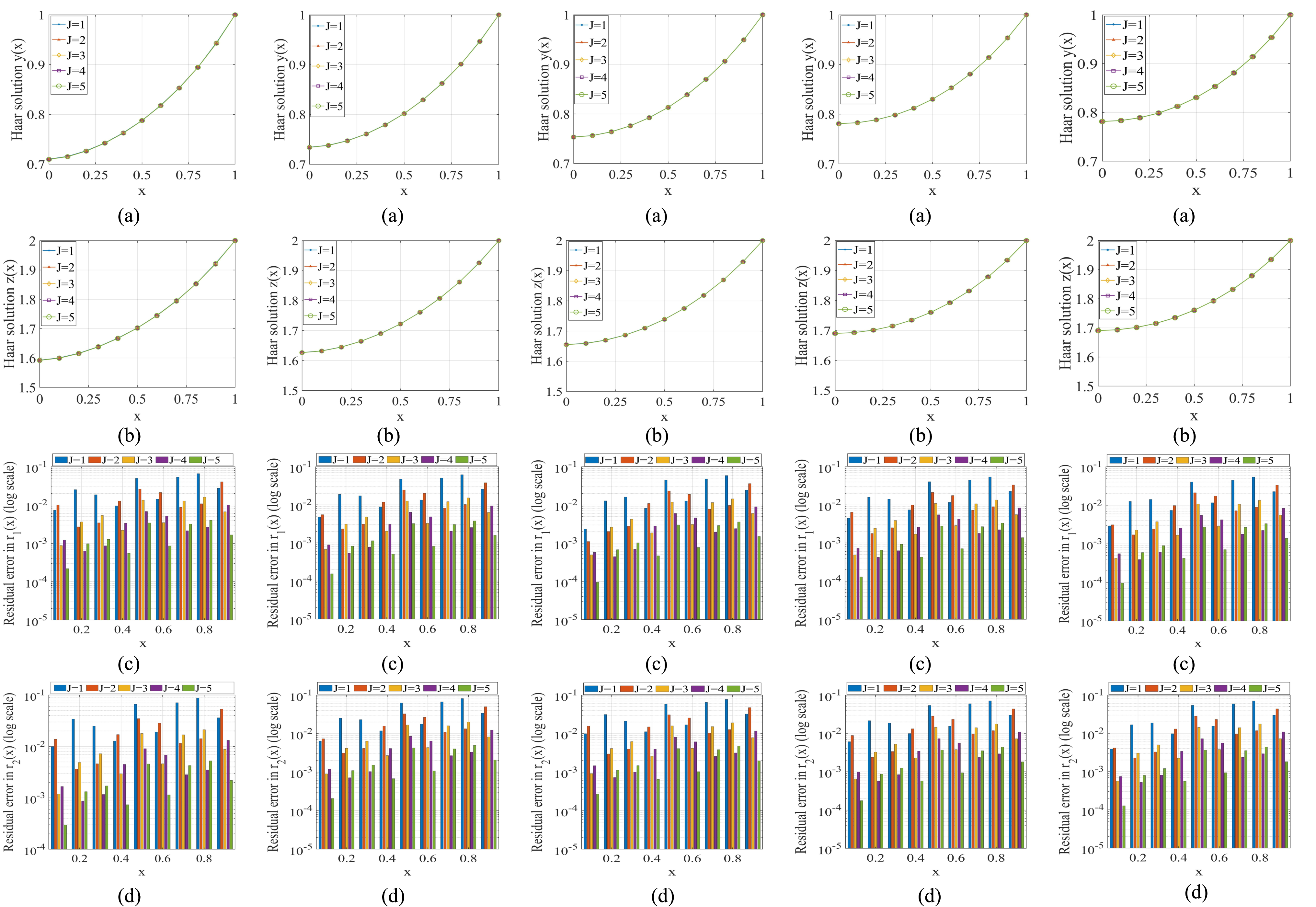} 
  \caption{Solution and error plots of example \ref{P4_cat_dif}, with varying parameters $\alpha_1$, $\beta_1$, $\alpha_2$, and $\beta_2$ arranged column-wise (from left to right): (a) $\alpha_1 = 1.61$, $\beta_1 = 0.62$, $\alpha_2 = 1.62$, $\beta_2 = 0.63$; (b) $\alpha_1 = 1.74$, $\beta_1 = 0.74$, $\alpha_2 = 1.75$, $\beta_2 = 0.75$; (c) $\alpha_1 = 1.85$, $\beta_1 = 0.84$, $\alpha_2 = 1.84$, $\beta_2 = 0.86$; (d) $\alpha_1 = 1.999$, $\beta_1 = 0.999$, $\alpha_2 = 1.99$, $\beta_2 = 0.999$; and (e) $\alpha_1 = 2$, $\beta_1 = 1$, $\alpha_2 = 2$, $\beta_2 = 1$.} 
  \label{P4_cat_dif_fig}
\end{figure}

\begin{table}[H]
  \begin{center}
    \caption{Total residual error of problem \ref{P4_cat_dif} at varying values of $\alpha_1,~  \beta_1,~\alpha_2~\text{and}~ \beta_2$ for various value of $J.$}
    \label{P4_table1_cat_dif}
    \begin{tabular}{|c|c|c|c|c|c|c|}
    \hline
     \multirow{2}{*}{x} & \multicolumn{3}{c|}{$(\alpha_1,\;\beta_1, \; \alpha_2,\; \beta_2) = (1.61,\; 0.62,\; 1.62,\; 0.63)$} &  \multicolumn{3}{c|}{$(\alpha_1,\; \beta_1,\; \alpha_2,\; \beta_2) = (1.74,\; 0.74,\; 1.75,\; 0.75)$}\\
\cline{2-7}
&$J=3$ &$J=4$  &$J=5$ &$J=3$ &$J=4$ & $J=5$\\
\hline
0.1 &	0.001474905 &	0.002069507 &	0.000369637 &	0.00113114 &	0.001487479 &	0.000259061\\
\hline
0.2 &	0.006086186 &	0.001069493 &	0.001645428 &	0.005167669 &	0.000897005 &	0.001372372\\
\hline
0.3 &	0.009008915 &	0.001450093 &	0.002134283 &	0.007975925 &	0.001289706 &	0.001904543\\
\hline
0.4 &	0.003667475 &	0.005595642 &	0.000921766 &	0.003368588 &	0.005128014 &	0.000846273\\
\hline
0.5 &	0.022361501 &	0.0113456 &	0.00571595 &	0.020981659 &	0.010626022 &	0.005347514\\
\hline
0.6 & 0.005773812 &	0.008548045 &	0.00143098 &	0.005431841 &	0.008059078 &	0.001348473\\
\hline
0.7 &	0.021218404 &	0.003549778 & 0.005335927 &	0.020059513 &	0.003355921 &	0.005044012\\
\hline
0.8 &	0.026592868 &	0.004415682 &	0.006607277 &	0.025115179 &	0.004172376 &	0.006247496\\
\hline
0.9 &	0.011047904 &	0.016519758 &	0.002752051 &	0.010394459 &	0.015558505 &	0.002592453\\
\hline
\textbf{$E$} &	\textbf{0.026592868} &	\textbf{0.016519758} &	\textbf{0.006607277} &	\textbf{0.025115179} &	\textbf{0.015558505} &	\textbf{0.006247496}\\
\hline
 \end{tabular}
  \end{center}
\end{table}

\begin{table}[H]
  \begin{center}
    \caption{Total residual error of problem \ref{P4_cat_dif} at varying values of $\alpha_1,~  \beta_1,~\alpha_2~\text{and}~ \beta_2$ for various value of $J.$}
    \label{P4_table2_cat_dif}
    \begin{tabular}{|c|c|c|c|c|c|c|}
    \hline
     \multirow{2}{*}{x} & \multicolumn{3}{c|}{$(\alpha_1,~\beta_1, ~\alpha_2,~\beta_2) = (1.85,~ 0.84,~ 1.84,~0.86)$} &  \multicolumn{3}{c|}{$(\alpha_1,~\beta_1, ~\alpha_2,~\beta_2) = (1.99,~ 0.999,~ 1.99,~0.999)$}\\
\cline{2-7}
&$J=3$ &$J=4$  &$J=5$ &$J=3$ &$J=4$ & $J=5$\\
\hline
0.1 & 0.001043474 &	0.001580631 &	0.000285226 &	0.000811452 &	0.001221587 &	0.000218482\\
\hline
0.2 & 0.004864753 &	0.000851715 &	0.001308136 &	0.004063223 &	0.000707617 &	0.001083906\\
\hline
0.3 & 0.007531036 &	0.001215191 &	0.001791164 &	0.00650489 &	0.001052725 &	0.001554292\\
\hline
0.4 & 0.003184477 &	0.004853347 &	0.000800194 &	0.002840758 &	0.004322004 &	0.000713301\\
\hline
0.5 & 0.01986764 &	0.01007493 &	0.005073756 &	0.018087284 &	0.009161733 &	0.004611055\\
\hline
0.6 & 0.005178906 &	0.00767315 &	0.001284615 &	0.004755726 &	0.007052825 &	0.001180399\\
\hline
0.7 & 0.019108768 &	0.003200382 &	0.004812836 &	0.017676957 &	0.002960286 &	0.004451387\\
\hline
0.8 & 0.023969195 &	0.003977493 &	0.005951558 &	0.022229476 &	0.003689323 &	0.005521236\\
\hline
0.9 & 0.009875031 &	0.01479991 &	0.002464505 &	0.009152846 &	0.013725496 &	0.002285511\\
\hline
\textbf{$E$} & \textbf{0.023969195} &	\textbf{0.01479991} &	\textbf{0.005951558} &	\textbf{0.022229476} &	\textbf{0.013725496} &	\textbf{0.005521236}\\
\hline
 \end{tabular}
  \end{center}
\end{table}

 \subsection{Fractional Concentration Of The Carbon Substrate And The Concentration Of Oxygen Problem}\label{P4_car_oxy}
 %Consider the fractional differential equation \eqref{P4_problem} with $k_1= 2, \gamma_1= 1,  f_1(\x, \y, \z)=  -1 + \frac{5 \y(\x) \z(\x)}{\Big(\frac{1}{10000} + \y(\x) \Big)\Big(\frac{1}{10000} + \z(\x) \Big)} + \frac{\frac{1}{10} \y(\x) \z(\x)}{\Big(\frac{1}{10000} + \y(\x) \Big)\Big(\frac{1}{10000} + \z(\x) \Big)},$ and $k_2= 2, \gamma_2= 1, f_2(\x, \y, \z)= \frac{\frac{1}{10} \y(\x) \z(\x)}{\Big(\frac{1}{10000} + \y(\x) \Big)\Big(\frac{1}{10000} + \z(\x) \Big)} + \frac{\frac{5}{100} \y(\x) \z(\x)}{\Big(\frac{1}{10000} + \y(\x) \Big)\Big(\frac{1}{10000} + \z(\x) \Big)},$ subject to $ \y'(0)=0, \z'(0)=0, \y(1)=1, \z(1)=1.$
We consider the fractional version of the concentration of carbon substrate and the concentration of oxygen problem,
\begin{equation}
\begin{aligned}
& \D^{\alpha_1} \y(\x) + \frac{2}{x} \D^{\beta_1} \y(\x) = -1 + \frac{5 \y(\x) \z(\x)}{\left(\frac{1}{10000} + \y(\x) \right)\left(\frac{1}{10000} + \z(\x) \right)} + \frac{\frac{1}{10} \y(\x) \z(\x)}{\left(\frac{1}{10000} + \y(\x) \right)\left(\frac{1}{10000} + \z(\x) \right)},\\
&\D^{\alpha_2} \z(\x) + \frac{2}{x} \D^{\beta_2} \z(\x) = \frac{\frac{1}{10} \y(\x) \z(\x)}{\left(\frac{1}{10000} + \y(\x) \right) \left(\frac{1}{10000} + \z(\x) \right)} + \frac{\frac{5}{100} \y(\x) \z(\x)}{\left(\frac{1}{10000} + \y(\x) \right) \left(\frac{1}{10000} + \z(\x) \right)},\\
& \y'(0)=0, ~ \z'(0)=0, \quad \y(1)=1, ~ \z(1)=1.
\end{aligned}
\end{equation}

\begin{itemize}
    \item[\textbf{(a)}] Figure \ref{P4_car_oxy_fig} visually represents our findings, with each column corresponding to a specific set of parameter values $\alpha_1, \beta_1, \alpha_2, \; \text{and} \; \beta_2.$ The first column presents the Haar solutions $\y(\x), \z(\x)$ alongside the residual errors in $r_1(\x)$ and $r_2(\x)$ for varying values of $J$ at $\alpha_1=1.62,~ \beta_1=0.62, ~\alpha_2=1.63,~ \beta_2=0.63$ respectively.
    \item[\textbf{(b)}]  The second column illustrates the corresponding solutions and residual errors for $\alpha_1=1.74,~ \beta_1=0.75,~ \alpha_2=1.75,~ \beta_2=0.73.$ The third column follows with solutions and residual errors at $\alpha_1=1.85,\; \beta_1=0.86,\; \alpha_2=1.86,\; \beta_2=0.85.$ The fourth column showcases results for $\alpha_1=1.999,\; \beta_1=0.99,\; \alpha_2=1.999,\; \beta_2=0.99.$
    \item[\textbf{(c)}] The final column presents the Haar solutions $\y(\x), ~\z(\x)$ alongside the residual error in $\y(\x)$ and $\z(\x)$ for various values of $J$ at $\alpha_1=2,\; \beta_1=1,\; \alpha_2=2,\; \beta_2=1.$ At $\alpha_1=\alpha_2=2, ~\beta_1 = \beta_2=1,$ this system of the equation reduces to the classical form of concentration of the carbon substrate and the concentration of oxygen equations, which is discussed in \cite{singh2020solving}.
    \item[\textbf{(d)}] To further support these observations table \ref{P4_table1_car_oxy} provides the computed total residual error for problem \ref{P4_car_oxy} for $J=3,\; 4,\; 5$ at  $\alpha_1=1.62,\; 1.74,~ \alpha_2=1.63,\;1.75,~ \beta_1=0.62, \;0.75, ~ \beta_2=0.63,\; 0.73.$ Similarly, Table \ref{P4_table2_car_oxy} presents the total residual error for $J=3,~ 4, ~5$ at  $\alpha_1=1.85, ~1.999, ~\alpha_2=1.86,~ 1.999,~ \beta_1=0.86,\; 0.99, ~\beta_2=0.85,~ 0.99.$ 
   \item[\textbf{(e)}] The effectiveness of the method is confirmed by the tables and figures, which show that the residual error steadily decreases as $J$ increases varying $\alpha_1, ~\beta_1,~ \alpha_2, ~ \text{and} ~ \beta_2$ and at $\alpha_1= \alpha_2=2,~\beta_1= \beta_2=1$ the residual error becomes negligible. The approach's accuracy and robustness are demonstrated by the comparable trend that is seen even when $J$ stays constant but $\alpha_1, ~\beta_1, ~\alpha_2,$ and $\beta_2 $ varying.
   \item[\textbf{(f)}] It is also observed that the coefficient matrix has a bounded condition number in all investigated instances. Additionally, the boundedness of the inverse matrix norm ensures the numerical stability of our method.
   \item[\textbf{(g)}] The final computed solution remains unaffected by variations in the initial solution, confirming that the proposed method is stable under all given set of parameters $(\alpha_1, ~\beta_1, ~\alpha_2, ~\beta_2)$.
   \end{itemize}
\begin{figure}[H]
  \centering
  \includegraphics[width=1.0\textwidth]{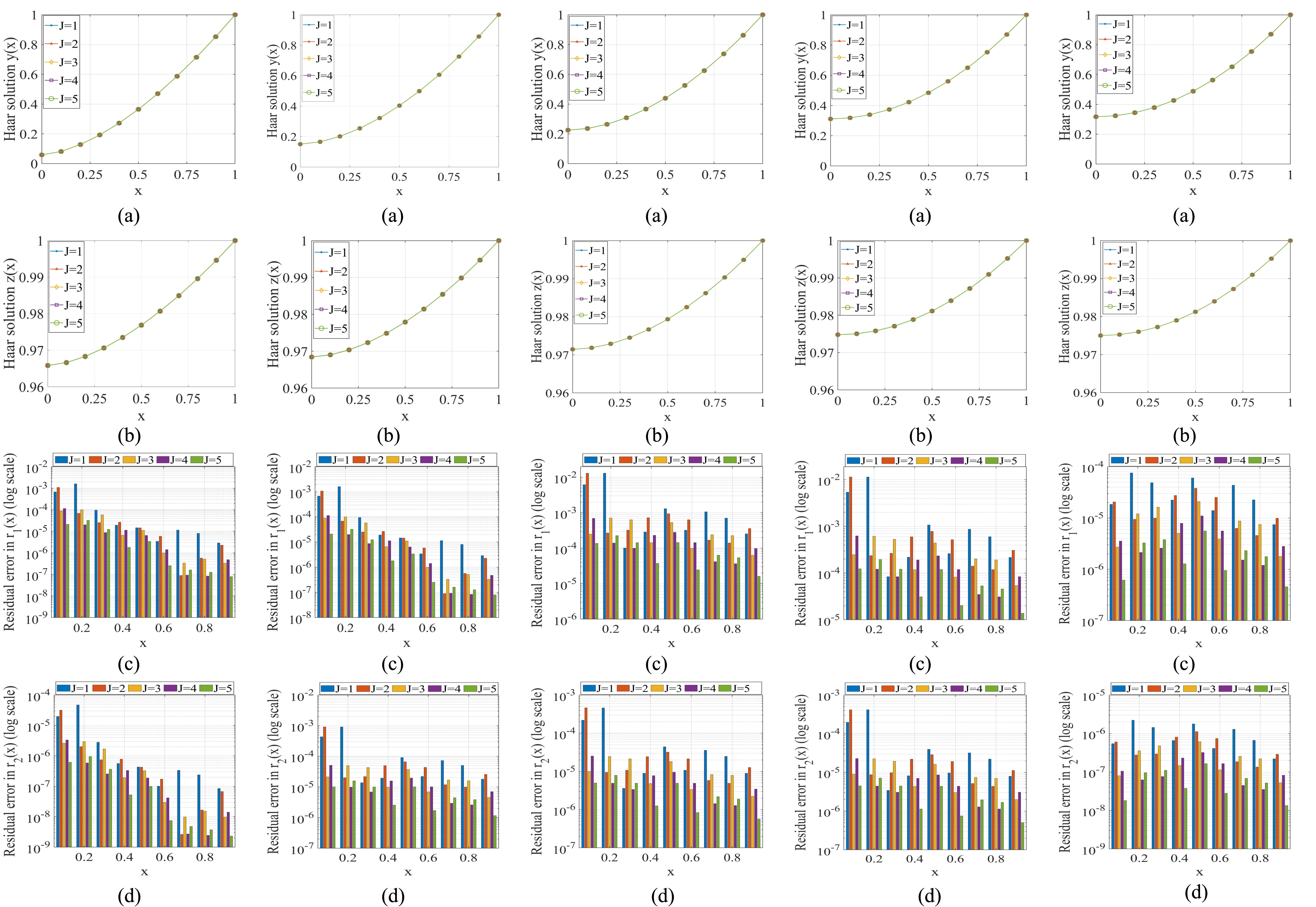} 
  \caption{Solution and error plots of example \ref{P4_car_oxy}, with varying parameters $\alpha_1$, $\beta_1$, $\alpha_2$, and $\beta_2$ arranged column-wise (from left to right): (a) $\alpha_1 = 1.62$, $\beta_1 = 0.62$, $\alpha_2 = 1.63$, $\beta_2 = 0.63$; (b) $\alpha_1 = 1.74$, $\beta_1 = 0.75$, $\alpha_2 = 1.75$, $\beta_2 = 0.73$; (c) $\alpha_1 = 1.85$, $\beta_1 = 0.86$, $\alpha_2 = 1.86$, $\beta_2 = 0.85$; (d) $\alpha_1 = 1.999$, $\beta_1 = 0.99$, $\alpha_2 = 1.999$, $\beta_2 = 0.99$; and (e) $\alpha_1 = 2$, $\beta_1 = 1$, $\alpha_2 = 2$, $\beta_2 = 1$.} 
  \label{P4_car_oxy_fig}
\end{figure}

\begin{table}[H]
  \begin{center}
    \caption{Total residual error of problem \ref{P4_car_oxy} at varying values of $\alpha_1,~  \beta_1,~\alpha_2~\text{and}~ \beta_2$ for various value of $J.$}
    \label{P4_table1_car_oxy}
    \begin{tabular}{|c|c|c|c|c|c|c|}
    \hline
     \multirow{2}{*}{x} & \multicolumn{3}{c|}{$(\alpha_1,~\beta_1, ~\alpha_2,~\beta_2) = (1.62, ~0.62, ~1.63,~0.63)$} &  \multicolumn{3}{c|}{$(\alpha_1,~\beta_1, ~\alpha_2,~\beta_2) = (1.74,~ 0.75, ~1.75,~0.73)$}\\
\cline{2-7}
&$J=3$ &$J=4$  &$J=5$ &$J=3$ &$J=4$ & $J=5$\\
\hline
0.1 &	8.97139E-05 &	0.000114013 &	2.11338E-05 &	0.000263448 &	0.000711335 &	0.000140386\\
\hline
0.2 &	0.000100115 &	1.99386E-05 &	3.23558E-05 &	0.00073694 &	0.000145201 &	0.000233376\\
\hline
0.3 &	5.77091E-05 &	8.74395E-06 &	1.24271E-05 &	0.000652171 &	0.000102411 &	0.000148087\\
\hline
0.4 &	6.5469E-06 &	1.13029E-05 &	1.79723E-06 &	0.000145553 &	0.000234338 &	3.79679E-05\\
\hline
0.5 &	1.11208E-05 &	6.42343E-06 &	3.41505E-06 &	0.00053632 &	0.000283636 &	0.000145605\\
\hline
0.6 &	1.00446E-06 &	1.41904E-06 &	2.54446E-07 &	9.95607E-05 &	0.000143927 &	2.45117E-05\\
\hline
0.7 &	3.35035E-07 &	9.2542E-08 &	1.62568E-07 &	0.000239534 &	4.14819E-05 &	6.33714E-05\\
\hline
0.8 &	5.16214E-07 &	8.34868E-08 &	1.27E-07 &	0.000225345 &	3.65178E-05 &	5.39631E-05\\
\hline
0.9 &	3.33228E-07 &	4.7882E-07 &	7.92003E-08 &	6.33718E-05 &	9.78497E-05 &	1.60947E-05\\
\hline
\textbf{$E$}	& \textbf{0.000100115} &	\textbf{0.000114013} &	\textbf{3.23558E-05} & \textbf{0.00073694} &	\textbf{0.000711335} &	\textbf{0.000233376}\\
\hline
 \end{tabular}
  \end{center}
\end{table}

\begin{table}[H]
  \begin{center}
    \caption{Total residual error of problem \ref{P4_car_oxy} at varying values of $\alpha_1,~\beta_1,~\alpha_2~\text{and}~ \beta_2$ for various value of $J.$}
    \label{P4_table2_car_oxy}
    \begin{tabular}{|c|c|c|c|c|c|c|}
    \hline
     \multirow{2}{*}{x} & \multicolumn{3}{c|}{$(\alpha_1,~\beta_1, ~\alpha_2,~\beta_2) = (1.85,~ 0.86,~ 1.86,~0.85)$} &  \multicolumn{3}{c|}{$(\alpha_1,~\beta_1, ~\alpha_2,~\beta_2) = (1.999,~ 0.99, ~1.999,~0.99)$}\\
\cline{2-7}
&$J=3$ &$J=4$  &$J=5$ &$J=3$ &$J=4$ & $J=5$\\
\hline
0.1 & 0.000251435 &	0.000695738 &	0.000137562 &	0.000248841 &	0.000626827 &	0.000123713\\
\hline
0.2 & 0.000713448 &	0.000140733 &	0.000226264 &	0.000616155 &	0.000121868 &	0.000195896\\
\hline
0.3 & 0.000635453 &	9.98914E-05 &	0.000144507 &	0.000530612 &	8.38865E-05 &	0.000121726\\
\hline
0.4 & 0.000143733 &	0.000231169 &	3.74694E-05 &	0.000118666 &	0.000190845 &	3.09928E-05\\
\hline
0.5 & 0.000535111 &	0.000282798 &	0.000145128 &	0.000443924 &	0.000234036 &	0.000119921\\
\hline
0.6 & 9.98045E-05 &	0.000144303 &	2.4571E-05 &	8.248E-05 &	0.000119746 &	2.0363E-05\\
\hline
0.7 & 0.000240893 &	4.17075E-05 &	6.37099E-05 &	0.000201981 &	3.48932E-05 &	5.32375E-05\\
\hline
0.8 & 0.000226818 &	3.67545E-05 &	5.43123E-05 &	0.000191144 &	3.10581E-05 &	4.59741E-05\\
\hline
0.9 & 6.37848E-05 &	9.84851E-05 &	1.61986E-05 &	5.45404E-05 &	8.40578E-05 &	1.38439E-05\\
\hline
\textbf{$E$}	&  \textbf{0.000713448} &	\textbf{0.000695738} &	\textbf{0.000226264}& 	\textbf{0.000616155} & 	\textbf{0.000626827} &	\textbf{0.000195896}\\
\hline
 \end{tabular}
  \end{center}
\end{table} 

\section{Conclusion}\label{P4_conclusion}
In this study, we introduce a new class of coupled fractional Lane-Emden equation \eqref{P4_problem} with conditions \eqref{P4_condition}. We examine these conditions in two different cases and perform numerical simulations to analyze the behavior of the system. We implement the fractional Haar wavelet collocation method combined with the Newton-Raphson method to analyze the proposed problem. We also establish the convergence of our approach, demonstrating its efficiency and reliability. To assess the method’s performance, we conduct five numerical experiments, illustrating its real-world applications.  Our results, presented through figures and tables, show that residual errors decrease as the maximum level of resolution $J$ increases while keeping the fractional order derivatives $\alpha_1,~ \beta_1,~ \alpha_2 \;\text{and}\; \beta_2$ fixed and similar trends occur when $\alpha_1,~ \beta_1,~ \alpha_2 \;\text{and}~ \beta_2$ vary and $J$ is fixed. At $\alpha_1 = \alpha_2 = 2$ and $\beta_1 = \beta_2 = 1$, the problem reduces back to the classical form of the coupled Lane Emden equations. In figures and tables, it can be observed that the residual error becomes negligible as we increase the value of the maximum level of resolution $J,$ at $\alpha_1 = \alpha_2 = 2$ and $\beta_1 = \beta_2 = 1,$ reinforcing the accuracy of the proposed method. These outcomes underscore the reliability of the fractional Haar wavelet collocation method in addressing such types of problems.  Our approach is novel and provides a foundation for further research in various areas where fractional differential equations play a key role.

\section*{Acknowledgement}
The first author is very much grateful to all the members of our research group at IIT Patna for their support and help. The work is financially supported to the first author by University Grants Commission(UGC) - (December 2019)/2019(NET/Joint CSIR-UGC), NTA Ref: no. 191620007135, New Delhi, India.

\section*{Conflict of interest}
The authors don't have any conflict of interest to disclose.
\bibliography{fractional_coupled_system}

\begin{thebibliography}{10}

\bibitem{AhmadLuca}
B.~Ahmad and R.~Luca.
\newblock Existence of solutions for a system of fractional differential
  equations with coupled nonlocal boundary conditions.
\newblock {\em Fractional Calculus and Applied Analysis}, 21(2):423--441, 2018.

\bibitem{Ahmad2023coupled}
A.~Ahmed, A.~Manal, A.~Bashir, and K.~N. Sotiris.
\newblock On a nonlinear coupled {C}aputo-type fractional differential system
  with coupled closed boundary conditions.
\newblock {\em AIMS Mathematics}, 8(8):17981--17995, 2023.

\bibitem{al2010numerical}
Q.~M. Al-Mdallal.
\newblock On the numerical solution of fractional {S}turm--{L}iouville
  problems.
\newblock {\em International Journal of Computer Mathematics},
  87(12):2837--2845, 2010.

\bibitem{ala2023numerical}
O.~Ala’yed, R.~Saadeh, and A.~Qazza.
\newblock Numerical solution for the system of {L}ane-{E}mden type equations
  using cubic {B}-spline method arising in engineering.
\newblock {\em AIMS Mathematics}, 8(6):14747--14766, 2023.

\bibitem{anderson1981complementary}
N.~Anderson and A.~M. Arthurs.
\newblock Complementary extremum principles for a nonlinear model of heat
  conduction in the human head.
\newblock {\em Bulletin of Mathematical Biology}, 43:341--346, 1981.

\bibitem{bai2004existence}
C.~Bai and J.~Fang.
\newblock The existence of a positive solution for a singular coupled system of
  nonlinear fractional differential equations.
\newblock {\em Applied Mathematics and Computation}, 150(3):611--621, 2004.

\bibitem{bai2005positive}
Z.~Bai and H.~L{\"u}.
\newblock Positive solutions for boundary value problem of nonlinear fractional
  differential equation.
\newblock {\em Journal of Mathematical Analysis and Applications},
  311(2):495--505, 2005.

\bibitem{cattani2004haar}
C.~Cattani.
\newblock {H}aar wavelets based technique in evolution problems.
\newblock {\em Proceedings-Estonian Academy Of Sciences, Physics, Mathematics},
  53(1):45--63, 2004.

\bibitem{chambre1952solution}
P.~L. Chambr{\'e}.
\newblock On the solution of the {P}oisson-{B}oltzmann equation with
  application to the theory of thermal explosions.
\newblock {\em The Journal of Chemical Physics}, 20(11):1795--1797, 1952.

\bibitem{chandrasekhar1939book}
S.~Chandrasekhar.
\newblock Book review: An introduction to the study of stellar structure, by
  {S}. {C}handrasekhar, 1939.

\bibitem{chandrasekhar1957introduction}
S.~Chandrasekhar.
\newblock {\em An introduction to the study of stellar structure}, volume~2.
\newblock Courier Corporation, 1957.

\bibitem{chen1997haar}
C.~F. Chen and C.~H. Hsiao.
\newblock {H}aar wavelet method for solving lumped and distributed-parameter
  systems.
\newblock {\em IEE Proceedings-Control Theory and Applications}, 144(1):87--94,
  1997.

\bibitem{chen2008numerical}
Y.~Chen and H.~An.
\newblock Numerical solutions of coupled burgers equations with time-and
  space-fractional derivatives.
\newblock {\em Applied Mathematics and Computation}, 200(1):87--95, 2008.

\bibitem{chen2010wavelet}
Y.~Chen, Y.~Wu, Y.~Cui, Z.~Wang, and D.~Jin.
\newblock Wavelet method for a class of fractional convection-diffusion
  equation with variable coefficients.
\newblock {\em Journal of Computational Science}, 1(3):146--149, 2010.

\bibitem{daftardar2004}
V.~Daftardar-Gejji and A~Babakhani.
\newblock Analysis of a system of fractional differential equations.
\newblock {\em Journal of Mathematical Analysis and Applications},
  293(2):511--522, 2004.

\bibitem{dehda2024numerical}
B.~Dehda, F.~Yazid, F.~S. Djeradi, K.~Zennir, K.~Bouhali, and T.~Radwan.
\newblock Numerical approach based on the {H}aar wavelet collocation method for
  solving a coupled system with the {C}aputo--{F}abrizio fractional derivative.
\newblock {\em Symmetry}, 16(6):713, 2024.

\bibitem{flockerzi2011coupled}
D.~Flockerzi and K.~Sundmacher.
\newblock On coupled {L}ane-{E}mden equations arising in dusty fluid models.
\newblock {\em Journal of Physics: Conference Series}, 268(1):012006, 2011.

\bibitem{gafiychuk2008mathematical}
V.~Gafiychuk, B.~Datsko, and V.~Meleshko.
\newblock Mathematical modeling of time fractional reaction--diffusion systems.
\newblock {\em Journal of Computational and Applied Mathematics},
  220(1-2):215--225, 2008.

\bibitem{ghosh2015solution}
U.~Ghosh, S.~Sarkar, and S.~Das.
\newblock Solution of system of linear fractional differential equations with
  modified derivative of {J}umarie type.
\newblock {\em arXiv preprint arXiv:1510.00385}, 2015.

\bibitem{hajji2014efficient}
M.~A. Hajji, Q.~M. Al-Mdallal, and F.~M. Allan.
\newblock An efficient algorithm for solving higher-order fractional
  {S}turm--{L}iouville eigenvalue problems.
\newblock {\em Journal of Computational Physics}, 272:550--558, 2014.

\bibitem{hariharan2009haar}
G.~Hariharan, K.~Kannan, and K.~R. Sharma.
\newblock {H}aar wavelet method for solving fisher’s equation.
\newblock {\em Applied Mathematics and Computation}, 211(2):284--292, 2009.

\bibitem{higler2000nonequilibrium}
A.~Higler, R.~Krishna, and R.~Taylor.
\newblock Nonequilibrium modeling of reactive distillation: a dusty fluid model
  for heterogeneously catalyzed processes.
\newblock {\em Industrial and Engineering Chemistry Research},
  39(6):1596--1607, 2000.

\bibitem{ibrahim2007existence}
R.~W. Ibrahim and Momani S.
\newblock On the existence and uniqueness of solutions of a class of fractional
  differential equations.
\newblock {\em Journal of Mathematical Analysis and Applications},
  334(1):1--10, 2007.

\bibitem{Izadi_2024}
M.~Izadi, K.~J. Ansari, and H.~M. Srivastava.
\newblock A highly accurate and efficient genocchi-based spectral technique
  applied to singular fractional order boundary value problems.
\newblock {\em Mathematical Methods in the Applied Sciences}, 47(12), 2024.

\bibitem{jackson1977transport}
R.~Jackson.
\newblock {\em Transport in porous catalysts}, volume~4.
\newblock Elsevier Scientific Publishing Company, 1977.

\bibitem{lok2024}
L.~N. Kannaujiya, N.~Kumar, B.~Gupta, and A.~K. Verma.
\newblock Uniform and non-uniform {H}aar wavelets solutions on fractional
  $\alpha$-singular {L}ane-{E}mden equations.
\newblock {\em Functional Differential Equations}, pages 55--80, 2024.

\bibitem{kilbas2006theory}
A.~A. Kilbas, H.~M. Srivastava, and J.~J. Trujillo.
\newblock {\em Theory and applications of fractional differential equations},
  volume 204.
\newblock elsevier, 2006.

\bibitem{klimek2018exact}
M.~Klimek, M.~Ciesielski, and T.~Blaszczyk.
\newblock Exact and numerical solutions of the fractional {S}turm--{L}iouville
  problem.
\newblock {\em Fractional Calculus and Applied Analysis}, 21(1):45--71, 2018.

\bibitem{krishna1997maxwell}
R.~Krishna and J.~A Wesselingh.
\newblock The {M}axwell-{S}tefan approach to mass transfer.
\newblock {\em Chemical Engineering Science}, 52(6):861--911, 1997.

\bibitem{narendra2025}
N.~Kumar, L.~N. Kannaujiya, and A.~K. Verma.
\newblock Uniform {H}aar wavelet solutions for fractional regular
  $\alpha$-singular {BVP}s modeling human head heat conduction under febrifuge
  effects.
\newblock {\em Mathematical Methods in the Applied Sciences}, pages 1--15,
  2025.

\bibitem{kumar2023hybrid}
N.~Kumar, D.~Tiwari, A.~K. Verma, and C.~Cattani.
\newblock Hybrid model for the optimal numerical solution of nonlinear ordinary
  differential systems.
\newblock {\em Computational and Applied Mathematics}, 42(8):322, 2023.

\bibitem{kumari2024novel}
S.~Kumari, L.~N. Kannaujiya, N.~Kumar, A.~K. Verma, and R.~P. Agarwal.
\newblock A novel hybrid variation iteration method and eigenvalues of
  fractional order singular eigenvalue problems.
\newblock {\em Journal of Mathematical Chemistry}, pages 1--22, 2024.

\bibitem{lee2024novel}
S.~Lee, H.~Kim, and B.~Jang.
\newblock A novel numerical method for solving nonlinear fractional-order
  differential equations and its applications.
\newblock {\em Fractal and Fractional}, 8(1):65, 2024.

\bibitem{lepik2005numerical}
{\"U}.~Lepik.
\newblock Numerical solution of differential equations using {H}aar wavelets.
\newblock {\em Mathematics and Computers in Simulation}, 68(2):127--143, 2005.

\bibitem{lepik2007application}
{\"U}.~Lepik.
\newblock Application of the {H}aar wavelet transform to solving integral and
  differential equations.
\newblock {\em Proceedings of the Estonian Academy of Sciences, Physics,
  Mathematics}, 56(1):28--46, 2007.

\bibitem{lepik2007numerical}
{\"U}.~Lepik.
\newblock Numerical solution of evolution equations by the {H}aar wavelet
  method.
\newblock {\em Applied Mathematics and Computation}, 185(1):695--704, 2007.

\bibitem{lepik2014haar}
{\"U}.~Lepik and H.~Hein.
\newblock {H}aar wavelets.
\newblock In {\em {H}aar wavelets: with applications}, pages 7--20. Springer,
  2014.

\bibitem{liu2017bifurcation}
L.~Liu, F.~Sun, X.~Zhang, and Y.~Wu.
\newblock Bifurcation analysis for a singular differential system with two
  parameters via to topological degree theory.
\newblock {\em Nonlinear Analysis: Modelling and Control}, 22(1):31--50, 2017.

\bibitem{liu2016extremal}
S.~Liu and H.~Li.
\newblock Extremal system of solutions for a coupled system of nonlinear
  fractional differential equations by monotone iterative method.
\newblock {\em Journal of Nonlinear Sciences and Applications},
  9(05):3310--3318, 2016.

\bibitem{mallat1989}
S.~G. Mallat.
\newblock Multiresolution approximations and wavelet orthonormal bases of $l^2
  (\mathbb{R})$.
\newblock {\em Transactions of The American Mathematical Society},
  315(1):69--87, 1989.

\bibitem{miller1993introduction}
K.~S. Miller and B.~Ross.
\newblock {\em An introduction to the fractional calculus and fractional
  differential equations}.
\newblock Wiley, 1993.

\bibitem{odibat2006approximations}
Z.~Odibat.
\newblock Approximations of fractional integrals and {C}aputo fractional
  derivatives.
\newblock {\em Applied Mathematics and Computation}, 178(2):527--533, 2006.

\bibitem{parand2024neural}
K.~Parand, A.~A. Aghaei, S.~Kiani, T.~I. Zadeh, and Z.~Khosravi.
\newblock A neural network approach for solving nonlinear differential
  equations of {L}ane--{E}mden type.
\newblock {\em Engineering with Computers}, 40(2):953--969, 2024.

\bibitem{podlubny1999introduction}
I.~Podlubny.
\newblock An introduction to fractional derivatives, fractional differential
  equations, to methods of their solution and some of their applications.
\newblock {\em Mathematics in Science and Engineering}, 198:340, 1999.

\bibitem{priyadarshi2018wavelet}
G.~Priyadarshi and B~.V.~R. Kumar.
\newblock Wavelet galerkin method for fourth order linear and nonlinear
  differential equations.
\newblock {\em Applied Mathematics and Computation}, 327:8--21, 2018.

\bibitem{rach2014solving}
R.~Rach, J.-S. Duan, and A.~M. Wazwaz.
\newblock Solving coupled {L}ane--{E}mden boundary value problems in catalytic
  diffusion reactions by the {A}domian decomposition method.
\newblock {\em Journal of Mathematical Chemistry}, 52:255--267, 2014.

\bibitem{raja2018new}
M.~A.~Z. Raja, M.~Umar, Z.~Sabir, J.~Khan, and D.~Baleanu.
\newblock A new stochastic computing paradigm for the dynamics of nonlinear
  singular heat conduction model of the human head.
\newblock {\em The European Physical Journal Plus}, 133:1--21, 2018.

\bibitem{Rehman2011note}
M.~Rehman and R.~A. Khan.
\newblock A note on boundary value problems for a coupled system of fractional
  differential equations.
\newblock {\em Computers and Mathematics with Applications}, 61(9):2630--2637,
  2011.

\bibitem{saeed2017haar}
U.~Saeed.
\newblock {H}aar {A}domian method for the solution of fractional nonlinear
  {L}ane-{E}mden type equations arising in astrophysics.
\newblock {\em Taiwanese Journal of Mathematics}, 21(5):1175--1192, 2017.

\bibitem{SALESTEODORO2019195}
G.~{Sales Teodoro}, J.A. {Tenreiro Machado}, and E.~{Capelas de Oliveira}.
\newblock A review of definitions of fractional derivatives and other
  operators.
\newblock {\em Journal of Computational Physics}, 388:195--208, 2019.

\bibitem{scherer2011grunwald}
R.~Scherer, S.~L. Kalla, Y.~Tang, and J.~Huang.
\newblock The {G}r{\"u}nwald--{L}etnikov method for fractional differential
  equations.
\newblock {\em Computers \& Mathematics with Applications}, 62(3):902--917,
  2011.

\bibitem{shiralashetti2020haar}
S.~C. Shiralashetti and E.~Harishkumar.
\newblock {H}aar wavelet matrices for the numerical solution of system of
  ordinary differential equations.
\newblock {\em Malaya Journal of Matematik}, 1:144--147, 2020.

\bibitem{singh2020solving}
R.~Singh.
\newblock Solving coupled {L}ane-{E}mden equations by {G}reen’s function and
  decomposition technique.
\newblock {\em International Journal of Applied and Computational Mathematics},
  6(3):80, 2020.

\bibitem{su2009boundary}
X.~Su.
\newblock Boundary value problem for a coupled system of nonlinear fractional
  differential equations.
\newblock {\em Applied Mathematics Letters}, 22(1):64--69, 2009.

\bibitem{SUN_2021106732}
H.~Sun, L.~Mei, and Y.~Lin.
\newblock A new algorithm based on improved legendre orthonormal basis for
  solving second-order {BVP}s.
\newblock {\em Applied Mathematics Letters}, 112:106732, 2021.

\bibitem{sweldens1998}
W.~Sweldens.
\newblock The lifting scheme: A construction of second generation wavelets.
\newblock {\em SIAM Journal on Mathematical Analysis}, 29(2):511--546, 1998.

\bibitem{verma2020applications}
A.~K. Verma and S.~Kayenat.
\newblock Applications of modified mickens-type nsfd schemes to {L}ane--{E}mden
  equations.
\newblock {\em Computational and Applied Mathematics}, 39:1--25, 2020.

\bibitem{verma2021note}
A.~K. Verma, N.~Kumar, M.~Singh, and R.~P. Agarwal.
\newblock A note on variation iteration method with an application on
  {L}ane--{E}mden equations.
\newblock {\em Engineering Computations}, 38(10):3932--3943, 2021.

\bibitem{verma2021haar}
A.~K. Verma, N.~Kumar, and D.~Tiwari.
\newblock {H}aar wavelets collocation method for a system of nonlinear singular
  differential equations.
\newblock {\em Engineering Computations}, 38(2):659--698, 2021.

\bibitem{verma2020review}
A.~K. Verma, B.~Pandit, L.~Verma, and R.~P. Agarwal.
\newblock A review on a class of second order nonlinear singular bvps.
\newblock {\em Mathematics}, 8(7):1045, 2020.

\bibitem{wang2018nonlocal}
G.~Wang, K.~Pei, R.~P. Agarwal, L.~Zhang, and B.~Ahmad.
\newblock Nonlocal hadamard fractional boundary value problem with hadamard
  integral and discrete boundary conditions on a half-line.
\newblock {\em Journal of computational and applied mathematics}, 343:230--239,
  2018.

\bibitem{wang2020new}
K.~Wang.
\newblock A new fractional nonlinear singular heat conduction model for the
  human head considering the effect of febrifuge.
\newblock {\em The European Physical Journal Plus}, 135(11):871, 2020.

\bibitem{wang2010new}
Y.~Wang, X.~Chen, Y.~He, and Zh. He.
\newblock New decoupled wavelet bases for multiresolution structural analysis.
\newblock {\em Structural Engineering and Mechanics, An International Journal},
  35(2):175--190, 2010.

\bibitem{xie2023effective}
L.~J. Xie.
\newblock An effective algorithm for solving a system of {L}ane-{E}mden
  equations arising from catalytic diffusion reactions.
\newblock {\em Engineering Letters}, 31(4), 2023.

\bibitem{xu2014iterative}
N.~Xu and W.~Liu.
\newblock Iterative solutions for a coupled system of fractional
  differential--integral equations with two-point boundary conditions.
\newblock {\em Applied Mathematics and Computation}, 244:903--911, 2014.

\bibitem{zhai2018unique}
C.~Zhai and R.~Jiang.
\newblock Unique solutions for a new coupled system of fractional differential
  equations.
\newblock {\em Advances in Difference Equations}, 2018(1):1--12, 2018.

\bibitem{zhai2019coupled}
C.~Zhai and J.~Ren.
\newblock A coupled system of fractional differential equations on the
  half-line.
\newblock {\em Boundary Value Problems}, 2019:1--22, 2019.

\bibitem{zhang2016existence}
H.~Zhang, Y.~Li, and W.~Lu.
\newblock Existence and uniqueness of solutions for a coupled system of
  nonlinear fractional differential equations with fractional integral boundary
  conditions.
\newblock {\em Journal of Nonlinear Sciences and Applications},
  9(05):2434--2447, 2016.

\bibitem{zhang2006existence}
S.~Zhang.
\newblock Existence of solution for a boundary value problem of fractional
  order.
\newblock {\em Acta Mathematica Scientia}, 26(2):220--228, 2006.

\bibitem{zhang2008eigenvalues}
X.~Zhang, L.~Liu, and H.~Zou.
\newblock Eigenvalues of fourth-order singular {S}turm--{L}iouville boundary
  value problems.
\newblock {\em Nonlinear Analysis: Theory, Methods and Applications},
  68(2):384--392, 2008.

\end{thebibliography}
\bibliographystyle{plain}

\end{document}